\documentclass[12pt]{article}
\usepackage[top=3cm,bottom=3cm,left=3.5cm,right=3.5cm]{geometry}
\usepackage[ansinew]{inputenc}
\usepackage{amssymb,amsfonts,amsmath,amsthm,color}
\usepackage{graphicx}

\newtheorem{Theorem}{Theorem}[section]
\newtheorem{Lemma}{Lemma}[section]

\newtheorem{Remark}{Remark}[section]
\newtheorem{Definition}{Definition}[section] 

\begin{document}
	
\title{\Large \bf Dynamics of Riemann waves with sharp measure-controlled damping}

\author{\normalsize M. M. Cavalcanti, T. F. Ma\thanks{Corresponding author.} , P. Mar\'{\i}n-Rubio, P. N. Seminario-Huertas}

\date{}

\maketitle

\begin{abstract} 
This paper is concerned with locally damped semilinear wave equations defined on compact Riemannian manifolds with boundary. We present a construction of measure-controlled damping regions which are sharp in the sense that their summed interior and boundary measures are arbitrarily small. The construction of this class of open sets is purely geometric and allows us to prove a new observability inequality in terms of potential energy rather than the usual one with kinetic energy. A unique continuation property is also proved. Then, in three-dimension spaces, we establish the existence of finite dimensional smooth global attractors for a class of wave equations with nonlinear damping and $C^1$-forces with critical Sobolev growth. In addition, by means of an obstacle control condition, we show that our class of measure-controlled regions satisfies the well-known geometric control condition (GCC). Therefore, many of known results for the stabilization of wave equations hold true in the present context.
\end{abstract}



\tableofcontents

\section{Introduction} 
Let $(M,{\bf g})$ be a compact Riemannian manifold with smooth boundary $\partial M$ and metric ${\bf g}$. In order to place our goals and results, let us firstly consider the linear wave equation 
$$
\left\{ 
\begin{array}{lll}
\partial_t^2 u - \Delta u + \chi_{\omega} \partial_t u = 0 \;\; \mbox{in} \;\; M \times \mathbb{R}^{+}, \smallskip \\
u = 0 \;\; \mbox{on} \;\; \partial M, \smallskip \\ 
u(0) = u_0, \;\; \partial_t u(0) = u_1, 
\end{array}
\right.
$$  
where $\Delta$ is the Laplace-Beltrami operator on $M$ and $\chi_{\omega}$ is the characteristic function of an open subset $\omega$ of $M$. 
The energy of the system is given by 
$$
E = \frac{1}{2} \int_{M} \left(| \partial_t u |^2 + | \nabla u |^2 \right) dx,   
$$ 
where $\nabla$ stands for the Levi-Civita connection on $M$. 

\medskip

It is well known that the energy $E(t)$ decays exponentially to zero if and only if $\omega$ satisfies (GCC), the geometric control condition, 
a sharp result by Bardos, Lebeau and Rauch \cite{blr} (cf. Burq and Gerard \cite{burq-gerard}). This condition asserts that there exists $T_0>0$ such that any generalized geodesic traveling with speed 1 hits $\omega$ before an elapsed time $T_0$. The idea that all geodesics must meet a control region $\omega$ was earlier considered for manifolds without boundary by Rauch and Taylor \cite{rt-indiana,rt-cpam} that goes back to Ralston \cite{ralston} in a Euclidean setting. 
A distinguished feature is that $\omega$ can be chosen with arbitrarily small volume ${\rm meas}_{M}(\omega)$.

\medskip 

One of our concerns is to control the measure of such observation region $\omega$. Let us consider a simple example with flat geometry. In Figure \ref{fig-square}, 
$M$ is a square of side $1$. It is clear that both regions $\omega_1$ and $\omega_2$ satisfy (GCC). Although ${\rm meas}_{M}(\omega_1)$ can be taken arbitrarily small, its boundary measure is ${\rm meas}_{\partial M}(\omega_1 \cap \partial M) > 2$. 
On the other hand, for $\omega_2$, the summed interior and boundary measure ${\rm meas}_{M}(\omega_2) + {\rm meas}_{\partial M}(\omega_2 \cap \partial M)$ can be taken arbitrarily small. Here, we say that an open subset $\omega$ of $M$ is $\varepsilon$-controllable (in measure) if given $\varepsilon > 0$, 
$$
{\rm meas}_M (\omega) + {\rm meas}_{\partial M} (\omega \cap \partial M) <\varepsilon.
$$
The question of whether the boundary measure ${\rm meas}_{\partial M}$ of a set $\omega$ satisfying (GCC) 
can be arbitrarily small was studied by Cavalcanti et al. \cite{cavalcanti1,cavalcanti2}. 

\begin{figure}[htb] 
\begin{center}
\includegraphics[scale=0.5]{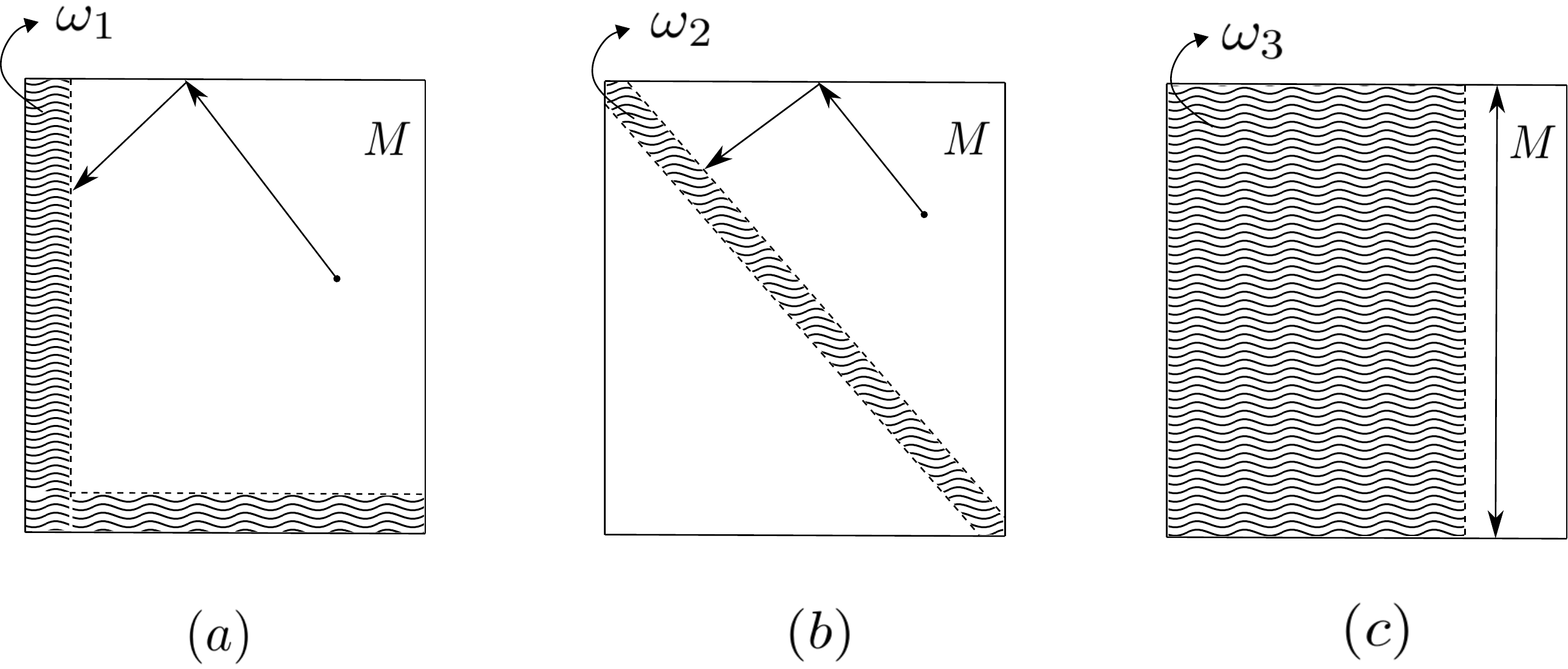} 
\end{center}
\caption{\small Let $M$ be the unitary square. It is easy to see that both control regions $\omega_1,\omega_2$ satisfy (GCC). 
The region $\omega_2$ is $\varepsilon$-controllable but $\omega_1$ is not because ${\rm meas}_{\partial M}(\omega_1) > 2$. 
The region $\omega_3$ does not satisfy (GCC) since $M$ possesses trapped (vertical) rays. }
\label{fig-square}
\end{figure}

\medskip 

We are interested in the long-time dynamics of semilinear waves with damping mechanism effective only in an $\varepsilon$-controllable region. 
As it is well-known (cf. \cite{clt,fz}) we shall need an appropriate unique continuation property and observability inequalities. 
We recall that the method used by Bardos, Lebeau and Rauch \cite{blr} combines fine results on propagation of singularities 
by Melrose and Sjostrand \cite{MS1,MS2} and microlocal analysis. Their arguments require that the solutions have higher regularity. 
Keeping in mind that we consider solutions of semilinear wave equations with $H^2(M)$-regularity we shall use another approach. 
Indeed, we follow in part the ideas developed in Cavalcanti et al. \cite{cavalcanti2} which is based on the results by Triggiani and Yao \cite{ty2002} 
and Lasiecka and Tataru \cite{Lasiecka-Tataru}. Their arguments use the concept of scape vector fields (cf. \cite{LTZ,miller,peng}) and    
aim to construct a control region by dividing the boundary $\partial M$ with respect to the sign of 
$\langle H , \nu \rangle$, where $\nu$ is the unit outward normal and $H$ is a strategic vector field.
This method requires less regularity and allows Carleman estimates.    

Our main contributions in this paper can be summarized as follows.  

\bigskip

$(i)$ Firstly, in Theorem \ref{thm-geometry}, we present the construction of a 
scape potential function $d$ defined in a part $V$ of the manifold $M$ in such a way that $M \setminus V$ is arbitrary 
small. This allows us to define a control/damping region $\omega \supset \overline{M \setminus V}$ that is 
$\varepsilon$-controllable. Our construction of $V,d,\omega$ is purely geometric and contrasts in great measure 
the presentation given in \cite{cavalcanti1,cavalcanti2} which is specific to a particular wave equation. Based on this construction we define 
the class of {\it admissible} $\varepsilon$-controllable sets. Indeed, in Theorem \ref{thm-overlapp}, we show the decomposition of 
an admissible set $\omega$ in overlapping sub-domains so that Carleman estimates can be performed in order to get observability and unique continuation.  

\medskip

$(ii)$ In Section \ref{subsec-gcc} we prove that our admissible $\varepsilon$-controllable sets satisfy (GCC). 
This is done by defining a new {\it obstacle} condition (Definition \ref{def-obstacle}) which is motivated 
by earlier ideas in \cite{Lasiecka-Tataru,miller}. Therefore many known results for control and stabilization  
of wave equations under (GCC) can be extended to the context of sharp measure-controlled damping region.

\medskip

$(iii)$ In Theorem \ref{thm-TYOVER} we revisit a controllability result based on Carleman estimates  
by Triggiani and Yao \cite{ty2002}. Then we use it to obtain (in-one-shot) 
observability and unique continuation for a linear wave equation plus a potential, locally damped in an  
{\it admissible} $\varepsilon$-controllable set. From this, through of a co-area property, we prove a new observability inequality 
of the form 
$$
\int_0^T \int_{\omega} |\nabla u|^2\, dxdt \ge C_{T} \big( E(0) + E(T) \big),        
$$
which is given in Theorem \ref{thm-ucp}. 
It turns out that this observability inequality suits remarkably the methods of quasi-stability by Chueshov and Lasiecka \cite{chueshov-book,CL-yellow}  
to study long-time dynamics of critical semilinear wave equations.  

\medskip 

$(iv)$ Let $M$ be a three-dimensional manifold with boundary. We study the long-time dynamics of semilinear wave equations 
$$
\partial_{t}^2 u - \Delta u + a(x) g(\partial_{t} u) + f(u) = 0 \;\; \mbox{in} \;\; M \times \mathbb{R}^{+}, 
$$   
with Dirichlet boundary condition and initial data in $\mathcal{H}=H^1_0(M) \times L^2(M)$. The nonlinear damping $g(\partial_t u)$ is globally Lipschitz,  because we seek finite dimensional attractors, and $f(u)$ may have critical Sobolev growth, namely $|f(u)| \approx |u|^3$. 
Both $f,g$ are required to have $C^1$-regularity only. 
	Then by combining observability inequality (Theorem \ref{thm-ucp}) and the recent theory of quasi-stable systems \cite{chueshov-book,CL-yellow}, 
	we establish the existence of regular finite dimensional attractors 
	by assuming $a(x)\ge a_0 >0$ on some {\it admissible} $\varepsilon$-controllable region $\omega$. 
	Detailed assumptions and proofs are presented in Section \ref{sec-dynamics}. To our best knowledge, comparable results were only proved earlier by 
	Chueshov, Lasiecka and Toundykov \cite{clt}, in an Euclidean setting with $f \in C^2$ and $\omega$ satisfying a geometric observability condition. 
	We notice that recently Jolly and Laurent \cite{joly-laurent} 
	proved the existence of global attractors for supercritical wave equations ($|f(u)| \approx |u|^{5-\epsilon}$) with linear damping 
	$\gamma(x) \partial_t u$ effective in a region $\omega$ satisfying (GCC). They arguments are based on a proper version of a unique continuation 
	property by Robbiano and Zuily \cite{robbiano-zuily}, that requires $f$ to be analytic. Fractal dimension and regularity of attractors 
	were not discussed but their results include unbounded domains and domains without boundary. As observed above, in the case of compact manifold 
	with boundary, their results can be extended to the framework of sharp measure-controlled damping region.

\medskip

In the present paper we only use standard concepts and notations on Riemannian geometry. 
For details we refer the reader to, for instance, do Carmo \cite{doCarmo} and Chavel \cite{cha}. 
With respect to Sobolev spaces on manifolds we refer the reader to Hebey \cite{hebey} and Taylor \cite{taylor}.
 
\section{Geometry for sharp measure control}
\label{sec-geometry} 
\setcounter{equation}{0}

\subsection{Sharp measure control condition} 
 
\begin{Definition} \label{epsilon}
	We say that a measurable subset $\omega$ of $M$ is $\varepsilon$-controllable (in measure) if given $\varepsilon> 0$,   
$$
{\rm meas}_M (\omega) + {\rm meas}_{\partial M} (\omega \cap \partial M) <\varepsilon,
$$
where ${\rm meas}_{A}(B)$ represents the measure of $B$ with respect to the Lebesgue measure defined in $A$. The class of $\varepsilon$-controllable sets of $M$ is denoted by $\chi_{\varepsilon}(M)$.
\end{Definition}

\begin{Remark} 
\rm We have the following properties for $\varepsilon$-controllable sets: 
Let $\varepsilon, \varepsilon_i>0$, $i=1,...,k$, then:
\begin{itemize}
	\item If $\omega_j \in \chi_{\epsilon_j}(M)$ then $\bigcup^k_{j=1} \omega_j \in \chi_{\varepsilon_1+ ... +\varepsilon_k}(M)$, 
	\item The (arbitrary) intersection of elements of $ \chi_{\varepsilon}(M) $ is an element of $ \chi_{\varepsilon}(M)$,
	\item Any set with null measure with respect to the measure of $\partial M$ and $M$, belongs to $ \chi_{\varepsilon}(M)$,
	\item Given $\varepsilon'>0$ such that $\varepsilon<\varepsilon'$ then $\chi_{\varepsilon}(M)\subset \chi_{\varepsilon'}(M)$,
	\item Given $M \subset \widetilde{M}$, then $\chi_{\varepsilon}(M) \subset \chi_{\varepsilon}(\widetilde{M})$,
	\item Given $r \in \mathbb{R}$, $\omega \in \chi_{\varepsilon}(M)$ and $p \in M$ such that $r\omega+p:=\{rx+p \ : \ x \in \omega\} \subset M$, then $r\omega+p \in \chi_{|r|\varepsilon}(M)$.
\end{itemize}
They are standard properties of Lebesgue measure (e.g. \cite{folland}). \qed 
\end{Remark}

We begin with a slightly more general version of a geometric construction by Cavalcanti et al. \cite[Section 6]{cavalcanti2}. 

\begin{Theorem} \label{thm-geometry}
Let $(M,{\bf g})$ be a connected compact $N$-dimensional Riemannian manifold of class $C^{\infty}$ with smooth boundary $\partial M$. 
Then, given $\varepsilon>0$ and $\varepsilon_0 \in (0,\varepsilon)$, the following hold: 
\begin{enumerate}
\item There exists an open set $V \subset M$, with smooth boundary $\overline{\partial V \cap {\rm int}(M)}$, that intercepts $\partial M$ transversally and satisfies 
$$
\overline{M\setminus V} \in \chi_{\epsilon_0}(M). 
$$ 
\item There exists a function $d: M \to \mathbb{R}$ such that:
\begin{itemize}
\item[$(d1)$] $d \in C^{\infty}(M)$,
\item[$(d2)$] ${\rm Hess} \, d(X,X) \ge |X|_{\bf g}^2, \;\; \forall \, X \in T_x M, \;\; x \in \overline{V}$,   
\item[$(d3)$] $\inf _{V} |\nabla d|_{\bf g} >0$,  \;\; $\min_{\overline{V}}d>0$,
\item[$(d4)$] $\langle \nabla d, \nu \rangle_{\bf g} <0 \;\; \mbox{on} \;\; \partial M \cap \overline{V}$. 
\end{itemize}
\item There exists an open set $\omega \in \chi_{\varepsilon}(M)$ such that 
$$
\overline{M \backslash V} \subset \omega \;\; \mbox{and} \;\; \omega\cap V \in \chi_{\varepsilon-\varepsilon_0}(M). 
$$
	\end{enumerate} 
\end{Theorem}

\begin{proof}
We first construct $V$ and $d$ locally with respect to interior points, and boundary points. 
Then we obtain global existence of $V$ and $d$ by using the compactness of $M$. 
The arguments are based on \cite[Section 6]{cavalcanti2}.   

\medskip
{\it Step 1:} We prove for any $p \in {\rm int}(M)$ there exists a neighborhood $V_p$ of $p$ and a function $d:V_p \to \mathbb{R}$ such that  $(d1)$-$(d3)$ hold with $V=V_p$. Indeed, given $p \in {\rm int}(M)$, there is an orthonormal basis $(e_1, \dots, e_N)$ of $ T_p M$ 
and a coordinate system $(x_1, \dots, x_N)$ over a neighborhood $ V_p$ contained in some chart $(U,\psi)$ 
such that $\partial x_i (p) = e_i (p)$, $i = 1,\dots,N$. 
We define the function $d: V_p \to \mathbb{R}$ 
by setting 
\begin{equation} \label{dddd}
	d(q)=\frac{1}{2}\sum_{j=1}^{N}x^2_j(q)+m,
\end{equation}
for some $m>0$. It is clear that 
$$
	|\nabla d(p)|_{\bf g}>0, \;\; \Delta d(p)=N, \;\; \inf_{q \in V_p} d(q)\ge m > 0. 
$$
Since Christoffel symbols with respect to 
$(x_1, \dots, x_N)$ satisfy $\Gamma^k_{ij}(p) = 0$ (see e.g. \cite{doCarmo} for details),  
it follows that ${\rm Hess} \, d(X, Y) = {\bf g} (X,Y)$ for all $X,Y \in T_pM$, which implies that $ {\rm Hess} \, d (X, X) = |X|_{\bf g}^2> 0 $ for all $ X \in T_p M $. Taking $V_p \subset \subset U$ small enough, $(d3)$ is satisfied with $V = V_p$.  
Because the coordinate system is the same for any element in $V_p$, 
we can use the same function $d$ on $V_p$ so that
$$
{\rm Hess} \, d(X,X)=|X|_{\bf g}^2, \;\; X \in T_{q} M, \;\; q \in V_p, 
$$
which proves $(d2)$. 
	
\medskip
	
{\it Step 2:} We show that given $p \in \partial M$, there exists a neighborhood $V_p$ of $p$ 
with smooth boundary $\overline{\partial V_p \cap {\rm int}(M)}$ which intercepts $\partial M$ transversally 
and a function $d:V_p \to \mathbb{R}$ satisfying $(d1)$-$(d4)$ with $V=V_p$. Indeed, fixed $p \in \partial M$, 
there exist a Riemannian manifold $\widetilde{M}$ and an isometric immersion 
$f: M \to \widetilde{M}$ such that $\overline{f(M)} \subset {\rm int}(\widetilde{M})$ (see \cite[Lemma 6.4]{cavalcanti2}). 
Take an orthonormal basis $(e_1, \dots, e_N)$ of $T_p \widetilde {M}$ such that $\nu(p) = -e_1$ 
is the outward normal vector at point $p$ with respect to $\partial M$. Proceeding as in Step 1, 
taking $\widetilde {M}$ instead of $M$, we obtain a neighborhood $\widetilde {V'_p} \subset \widetilde{M}$ of $p$. 
Due to the regularity of $\partial \widetilde{V'_p} \cap \partial M$ there is an open set 
$\widetilde{V_p}$ compactly embedded in $\widetilde{V'_p}$ with $p \in \widetilde{V_p}$ such that $\nu(q)=-e_1$ for all 
$q \in \overline{\widetilde{V_p}} \cap \partial M$. Moreover, we define $d:\widetilde{V_p} \to \mathbb{R}$ such that
\begin{equation} \label{ddd} 
	d(q)=x_1(q)+\frac{1}{2}\sum_{j=1}^{N}x^2_j(q)+m,
\end{equation}
for some $m>0$. It is evident that $\inf_{q \in V_p}|\nabla d(q)|_{\bf g}>0$, $\inf_{q \in V_p}d(q)\ge m$, $ \Delta d(p) = N $, 
and $ {\rm Hess} \, d(X, Y) = {\bf g} (X, Y) $ for all $ X, Y \in T_pM $. Then, as before, 
$(d1)$-$(d3)$ holds for $V = \widetilde{V_p}$. Additionally,
$$
\langle \nabla d (q), \nu(q) \rangle_{\bf g} < 0, \;\; q \in \widetilde{V_p} \cap \partial M.
$$ 
Finally, as shown in Figure \ref{mar2}, we can find a neighborhood $V_p \subset \widetilde{V_p} \cap M$ 
such that $\overline {\partial V_p \cap {\rm int}(M)}$ intercepts $\partial M$ transversally, completing the goal of Step 2. 
	
\begin{figure}[htb] 
\begin{center}
\includegraphics[scale=0.20]{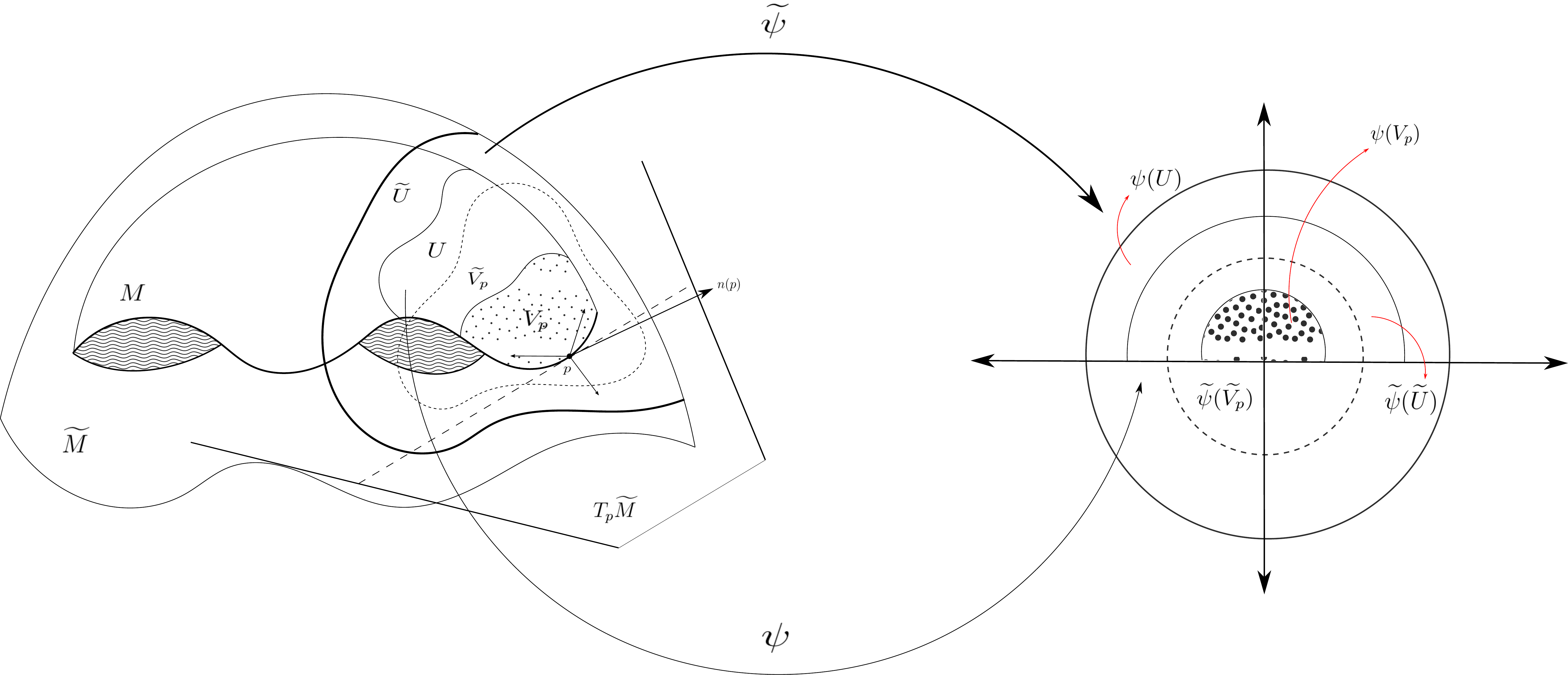} 
\end{center}
\caption{In the figure, $ (U, \psi) $ and $ (\widetilde{U}, \widetilde{\psi}) $ represent charts containing $p$, respectively in $M$ and $ \widetilde{M}$.  Note that $ \psi(V_p) \subset \psi(U) \subset \widetilde {\psi} (\widetilde{U}) $ so we can use the same coordinate system for every point in $V_p \subset M$.} \label{mar2}
\end{figure}
	
\medskip
	
{\it Step 3:} We recall that for relatively compact sets $A,B \subset M$ satisfying 
$$
dist (A,B) = \inf \{ dist(x,y) \, | \, x \in A \; y \in B \} > 0,  
$$ 
there exist open subsets $O_A \supset \supset A$ and $O_B \supset \supset B$ with smooth boundaries such that $dist(O_A, O_B)>0$. 
Moreover, there exists a smooth (cut-off) function $\rho : M \to \mathbb{R}$ such that 
$\rho|_{O_A}=1$, $\rho|_{O_B}=0$ and $\rho(M) \subset  [0, 1]$. See \cite[Lemma 6.9]{cavalcanti2}. 
The above sets $O_A$ and $O_B$ can be constructed, for any $\epsilon \in (0, dist(A,B)/3)$,  
such that 
\begin{equation} \label{oaob}
A \subset \subset O_A \subset \subset A_\epsilon \;\; \mbox{and} \;\; B \subset \subset O_B \subset \subset B_\epsilon ,
\end{equation} 
where
$$
A_{\epsilon}=\{x \in M \ | \ dist(x,A)<\epsilon \}, \ \ B_{\epsilon}=\{x \in M \ | \ dist(x,B)<\epsilon \},
$$
and $dist(x,Y)= \inf_{y \in Y} dist(x,y)$ with $dist(x,y)=|x-y|_{\bf g}$.

\medskip

{\it Step 4:} Conclusion: Repeating the strategy of Step 2, we can extend $M$ to a Riemannian manifold $\widetilde{M}$ such that, 
for each $p \in M$, one can choose a neighborhood $\widetilde{W_p}$ of $p$, 
and a function $d_p \in C^\infty(\widetilde{W_p})$ 
such that
\begin{itemize}
\item If $p \in {\rm int}(M)$, then choose $\widetilde{W_p}=V_p$ as in Step 1.
\item If $p \in \partial M$, then choose $\widetilde{W_p}=\widetilde{V_p} \subset {\rm int}(\widetilde{M})$ as in Step 2.
\end{itemize}
Then, 
due to the compactness of $M$, we can choose a finite sub-cover 
	$ \{\widetilde{W_j}\}^k_{j =1}$ of $ M $ such that if $ p \in \widetilde{W_j} $ for some $ j = 1 , ..., k $ denote by $ \widetilde{d_j} = d_{p}|_{W_j} $. Let $ B =\bigcup^ k_{j = 1} \partial \widetilde{W_j} \cap M $ where clearly $ M \setminus B $ is an open subset of $ M $. Denoting $ \left(M \setminus B \right) \cap \widetilde{W_1} $ as $ W_1 $ and $ \left (M \setminus B \right) \cap \left (\widetilde{W_j} \setminus \bigcup^{j-1}_{i = 1} \widetilde{W_i} \right) $ as $ W_j $ for $ j = 2, ..., k $, it gives that $ M \setminus B = \bigcup_{j = 1}^{k} W_j$.
		
	On the other hand, fixed $ \varepsilon> 0 $, for each $ \varepsilon_0 \in (0, \varepsilon) $ and $ W_j $ with $ j = 1, ..., k $, 
	it is possible to build an open $ U_j $ of $ M $ such that $ \overline{U_j} \subset W_j $ and $ {\rm meas}_M (W_j \setminus U_j) <\frac{\varepsilon_0}{2k} $ (see Figure \ref{mar4}). In addition, if $ W_j $ is a neighborhood of a boundary point of $ M $, then we can take $U_j$ such that 
	$ {\rm meas}_{\partial M} (\partial M \cap (W_j \setminus U_j)) <\frac{\varepsilon_0}{2k} $ 
	(see \cite[Lemma 6.7]{cavalcanti2} for more details).

\begin{figure}[htb] 
\begin{center}
\includegraphics[scale=0.4]{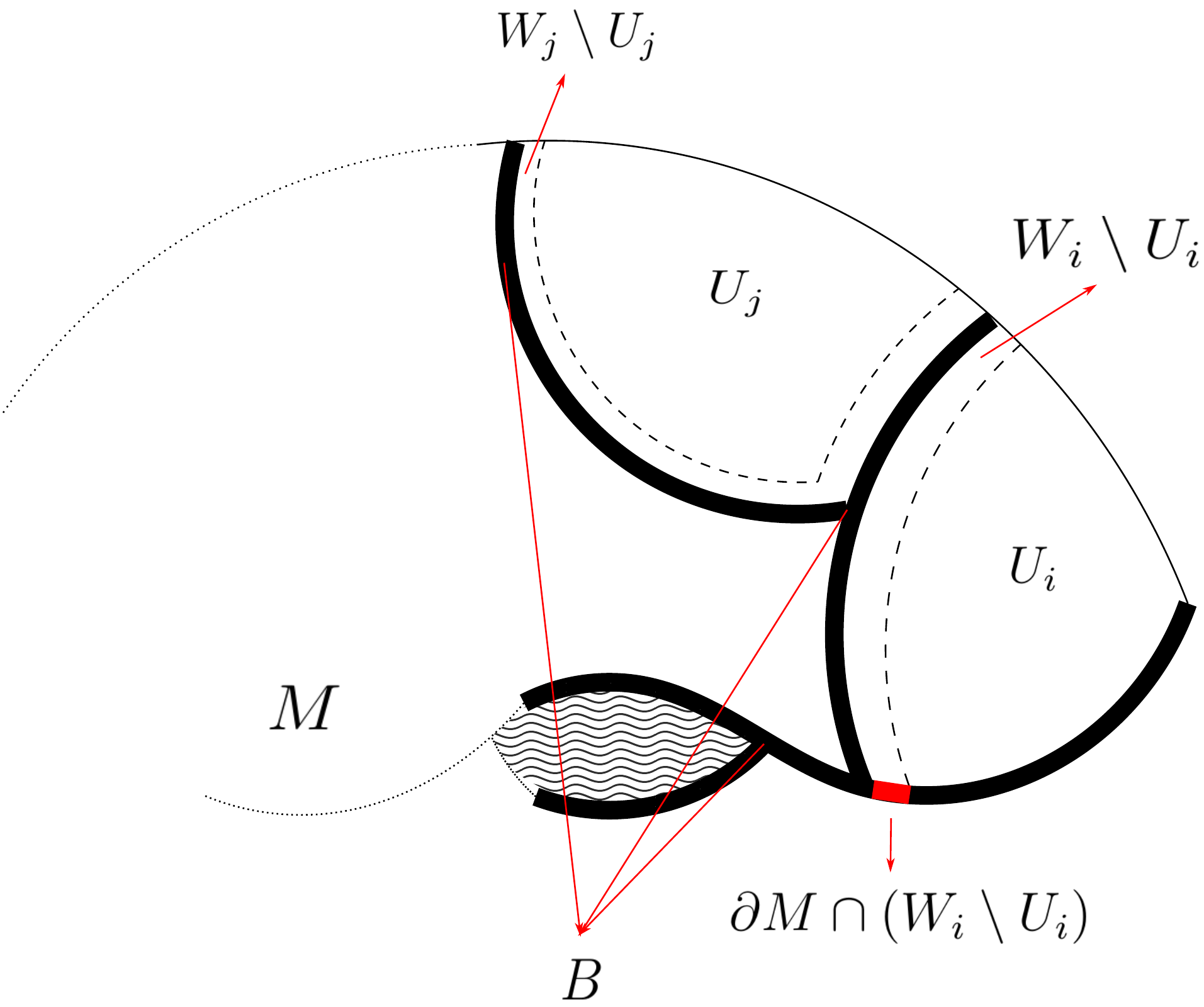} 
\end{center}
\caption{Given $i = 1,...,k $, we have total control over the interior measure of $ W_i \setminus U_i $ 
and boundary measure of $ \partial M \cap \left (W_i \setminus U_i \right) $, provided they are positive. Note that if $ \frac {\varepsilon_0} {2k} \ge \min \{{\rm meas}_M (W_i), {\rm meas}_{\partial M } (\partial M \cap W_i) \}> 0 $ it is possible to choose some $ 0 <\frac{\varepsilon'}{2k} <\{{\rm meas}_M(W_i), {\rm meas}_{\partial M}(\partial M \cap W_i) \} $ such that the measure of the aforementioned sets are less than $\frac{\varepsilon_0}{2k'}$ where $k'=\frac{k\varepsilon_0}{\varepsilon'}$   .} \label{mar4}
\end{figure}

	Because $ \overline{U_j} \subset W_j $, we can define $ d_j = \widetilde{d_j}|_{U_j} $. Also, from the compactness of $ B $ and $ \overline{U_j} $,  there are numbers $ \delta_j> 0$, $j = 1, ..., k$, such that $ d(B, \overline{U_j}) = \delta_j $. Then by 
	Step 3, exist open sets $ V_j \supset \supset U_j $ and $ O_j \supset \supset M \setminus W_j $ of $ M $ with smooth boundaries, 
	and a function $ \rho_j: M \to \mathbb{R} $ such that $ \rho_j|_{V_j} = 1, \rho_j|_{O_j} = 0 $ and $ \rho_j(M) \subset [0,1] $. 
	Note that in view of \eqref{oaob}, we can construct $V_j$ such that $ V_j \subset W_j $, so that $ \{V_j \}^k_{j = 1} $ is a disjoint family of open subsets	and $ \widetilde{d_j} $ is defined on each $ V_j $.

	Note that if $ V_j $ is a neighborhood intersecting $ \partial M $, then it is possible to assume that $V_j$ has smooth boundary $ \overline{\partial V_j \cap {\rm int}(M)} $ that intercepts $ \partial M $ transversally. Thus, we define $ d_j = \widetilde{d_j}|_{\overline{V_j}}$, $ \rho = \sum_{j = 1}^{k} \rho_j $ 
	and
$$ 
V=\bigcup^k_{j=1} V_j,
$$
so that $\rho|_V=1$ and (\ref{dddd}) is satisfied.
	
For the construction of $ d $, it is enough to define
$$
d(x) = \left\{ 
\begin{array}{ll}
d_j(x)\rho(x) & \textrm{if} \; x \in W_j, \\
0 & \textrm{otherwise},
\end{array}\right.
$$
which clearly satisfy $(d1)$-$(d4)$. 
	
Finally, from the construction of $V$, there is an open set  
	$ \omega \supset \overline {M \setminus V}$ such that $ \omega \cap V \in \chi_{\varepsilon- \varepsilon_0} (M) $. 
	From (\ref{dddd}) we see that $ \omega $ is $\varepsilon$-controllable. 
The result is proved.
\end{proof}

\begin{Remark} \rm $(a)$ The choice of $ \varepsilon_0 \in (0, \varepsilon)$ in Theorem \ref{thm-geometry} is independent of any other condition, that is, the result is valid for any 
$\varepsilon_0 \in (0, \varepsilon)$. This value represents the measure that is to be granted to the set $\overline{M \setminus V}$. As we will see in Section \ref{sec-obs}, 
the damping must be effective in a neighborhood $\omega$ of $\overline{M \setminus V}$ 
in order to prove an observability inequality. $(b)$ We note that once $\varepsilon_0$ is chosen, the construction of $V$, $d$ and the choice of $\omega$ involve mainly 
three properties:
\begin{itemize}
\item $\omega$ is an open subset of $M$,	
\item $\left( \overline{M \setminus V}\right) \cup (\omega \cap V) = \omega$ with $\left( \overline{M \setminus V}\right) \cap (\omega \cap V) = \emptyset$, 
\item $\overline{M \setminus V} \in \chi_{\varepsilon_0}(M)$ and $\omega\cap V \in \chi_{\varepsilon-\varepsilon_0}(M)$.
\end{itemize}
We note that it is possible to build different sets $ \omega $ such that $ \omega \cap V \in \chi_{\varepsilon- \varepsilon_0}(M) $. \qed
\end{Remark}

This motivates the following definition. 

\begin{Definition} \label{admissible-epsilon} 
Given $\varepsilon >0$, the family   
\begin{equation} \label{Vd} 
[\omega_{\varepsilon}] = 
\Big\{ 
\omega \in \chi_{\varepsilon}(M) \, | \, \mbox{$\omega$ is given by Theorem \ref{thm-geometry} for some $\varepsilon_0 \in (0,\varepsilon)$} 
\Big\},
\end{equation}
is called the class of admissible $\varepsilon$-controllable regions.  
\end{Definition}
The above definition will be used to characterized the idea of a sharp measure-controlled damping region.

\subsection{Bridge to (GCC)} \label{subsec-gcc} 
In this section, we show in Theorem \ref{thm-result2} that our sharp measure-controlled damping region satisfies (GCC). 
To simplify a little our presentation we shall assume the reader is familiar with generalized geodesics 
on compact manifold with boundary. Details can be found in, for instance, \cite{blr,miller}. 
Let $(M,\bf{g})$ be a $N$-dimensional compact manifold with smooth boundary $\partial M$. To our purpose, a generalized geodesic (ray of geometric optics) is a continuous trajectory $t \mapsto \gamma(t)$ which behaves as a geodesic of speed 1 in ${\rm int}(M)$, with the following additional features: 
\begin{itemize}
\item If $\gamma(t)$ hits $\partial M$ transversally at time $ t_0 $, then either it reflects as a billiard ball or 
it escapes from $M$ for $t>t_0$.	
\item If $\gamma(t)$ hits $\partial M$ tangentially at time $t_0$, then eventually, either it returns to ${\rm int}(M)$ or it escapes from $M$ at time $t>t_1>t_0$. 
\end{itemize}
In general, the generalized geodesics are not uniquely defined. However, uniqueness can be observed under additional assumptions, for instance: the metric ${\bf g}$ and boundary $\partial M$ are real analytic, or ${\bf g}$ and $\partial M$ are $C^\infty$ 
and $\partial M$ does not have contacts of infinite order with its tangents. See for instance \cite{burq,LH}. 

\smallskip

In what follows, it may be convenient defining (GCC) with an explicit ``control'' time $T$ as in \cite{miller}.   

\begin{Definition} [GCC] \label{def-gcc} 
Let $\omega$ be an open set of $M$ and $T>0$. A pair $(\omega,T)$ satisfies the geometric control condition  
if every generalized geodesic of length greater than $T$ intersects $\omega$.
\end{Definition}

Now, we borrow some ideas and results from Miller \cite{miller}. First, the bicharacteristic condition of Bardos, Lebeau and Rauch \cite{blr} is presented in terms of generalized geodesics, namely (cf. \cite[Definition 2.2]{miller}).  

\begin{Definition} [Geodesic condition]\label{def-geodesic}
Let $\Gamma$ be an open region of the boundary $\partial M$ and $T>0$. A pair $(\Gamma,T)$ satisfies the geodesic condition 
if every generalized geodesic of length greater than $T$ escapes from $M$ through $\Gamma$.  
\end{Definition}  
We note that from above definition, any open set $\omega \subset M$ containing $\Gamma$ satisfies (GCC). 
Also, as discussed in \cite{miller}, on can think $\partial M \setminus \Gamma$ as a border ``obstacle'' that prevents  
the generalized geodesics to leave $M$. Then we see $\Gamma$ as a border ``hole'' that allows generalized geodesics to escape from $M$. 
Formally we have the following definition (cf. \cite[Definition 4.1]{miller}).  

\begin{Definition} [Escape potential condition] \label{def-EVF}
Suppose that metric {\bf g} is $C^2$ and the boundary $\partial M$ is $C^3$. 
Let $\Gamma$ be an open subset of $\partial M$ and $T>0$. One says the pair $(\Gamma,T)$ satisfies 
the escape potential condition if there is a $C^2$-function $d:M \to \mathbb{R}$ such that
\begin{itemize}		
\item $|\nabla d|_{\bf g}\le T/2$, \;\; $\forall \, x \in M$,
\item ${\rm Hess} \, d(X,X) \ge |X|^2_{\bf g}$, \;\; $\forall \, X \in TM$,
\item $\left\lbrace x \in \partial M \, | \, \langle \nabla d , \nu \rangle > 0 \right\rbrace  \subset \Gamma$.
\end{itemize}
\end{Definition}

Under above geometric interpretation we formalize the idea of an obstacle condition.

\begin{Definition} [Obstacle condition] \label{def-obstacle}
Let $\Gamma_0$ be a subset of $\partial M$ and $T>0$. We say the pair $(\Gamma_0,T)$ satisfies the {\it obstacle condition} 
if there is a $C^3$-function $d:M \to \mathbb{R}$ such that
\begin{itemize}
	\item $|\nabla d|_{\bf g}\le T/2$, \;\; $\forall \, x \in M$,
	\item ${\rm Hess} \, d(X,X)\ge |X|^2_{\bf g}$, \;\; $\forall \, X \in TM$,
	\item $\Gamma_0 \subset \left\lbrace x \in \partial M \, | \, \langle \nabla d , \nu \rangle \le 0 \right\rbrace$.
\end{itemize}
\end{Definition}

\begin{Remark} \label{remark-miller} \rm 
$(a)$ It was proved by Miller \cite[Proposition 4.2]{miller} that $(\Gamma,T)$ satisfying {\it escape potential condition} also satisfies the 
{\it geodesic condition}. The converse is true if one assumes uniqueness for generalized geodesics. 
$(b)$ Suppose that $(\Gamma_0,T)$ satisfies the {\it obstacle geometric condition} with respect to a escape function $d$. If in addition $\Gamma' := \left \{ x \in \partial M \, | \; \langle \nabla d , \nu \rangle > 0 \right\} \neq \emptyset$,  
then $(\Gamma,T)$ clearly satisfies the {\it escape potential condition}, for any open subset $\Gamma \subset \partial M$ 
containing $\Gamma'$. \qed
\end{Remark}

We are in position to establish our main result of Section \ref{subsec-gcc}. 

\begin{Theorem} \label{thm-result2} Given $\varepsilon>0$, under the hypotheses of Theorem \ref{thm-geometry}, 
for each admissible damping region $\omega \in [\omega_{\varepsilon}]$ there exists $T=T(\omega)>0$ such that the pair $(\omega,T)$ 
satisfies (GCC).
\end{Theorem}

\begin{proof}
Fixed $\omega \in [\omega_{\varepsilon}]$, by construction there is an open disconnection $V \subset M$ such that 
$V = \bigcup^k_{j=1} V_j$ where $V_j$ are smooth connected open sets of $M$. 
By construction, for each $j$, there is a smooth connected compact $N$-dimensional Riemannian sub-manifold $(\Omega_j,{\bf g})$ 
of $M$, with boundary $\partial \Omega_j$, such that:
\begin{itemize}
	\item $\Omega_j \subset V_j$ and $V_j \setminus \Omega_j \subset \omega$,
	\item $\partial \Omega_j \cap {\rm int}(M) \subset \omega$,
	\item $\partial \Omega_j \cap \partial M \subset \partial V_j$.
\end{itemize}
Let us define $d^j = d|_{\Omega_j}$ where $d: M \to \mathbb{R}$ is the smooth (escape) function given by Theorem \ref{thm-geometry}. 
Then ${\rm Hess} \, d^j (X,X) \ge |X|^2_{\bf g}$ for all $X \in T \Omega_j$. 
Also, $\Gamma^j_0 = \partial \Omega_j \cap \partial M$ is closed and   
$$
\Gamma^j_0 \subseteq \left\{ x \in \partial \Omega_j \, | \; \langle \nabla d^j , \nu \rangle \le 0 \right\}.  
$$
Hence there exists $T^j>0$ such that the pair $(\Gamma^j_0, T^j)$ satisfies the {\it obstacle condition} on $(\Omega_j,{\bf g}|_{\Omega_j})$. 

We are going to show that $(\omega \cap \Omega_j,T^j)$ satisfies (GCC). Indeed, using local coordinates, we can see that the unit normal $\nu$ has sign $-1$ on $\partial V_j \cap \partial M$ 
and sign $+1$ on $\partial \Omega_j \cap {\rm int}M$. This implies that  
$\Gamma' = \left\{ x \in \partial \Omega_j \, | \, \langle \nabla d^j , \nu \rangle > 0 \right\}$ is nonempty. Then the open set $\Gamma^j_1 = \partial \Omega_j \setminus \Gamma^j_0$ contains $\Gamma'$.  
From Definition \ref{def-EVF} the pair $(\Gamma^j_1,T^j)$ satisfies the {\it escape potential condition}.  
Consequently (cf. Remark \ref{remark-miller}) we infer that $(\Gamma^j_1,T^j)$ satisfies the {\it geodesic condition}. 
In particular, since $\Gamma^j_1 \subset \omega$, the pair $(\omega \cap \Omega_j, T^j)$ satisfies (GCC). 
	
To conclude, we show that $(\omega,T)$ satisfies (GCC) with $T= \max \{T^1, ...,T^k\}$. 
Let $t \mapsto \gamma(t)$ be a generalized geodesic in $M$, $t \in \mathbb{R}$. We have three possibilities: 
\begin{itemize}
\item If $\gamma(t) \in \Omega_j$, then $\gamma(t+T^j) \in \omega$ because $(\omega \cap \Omega^j , T^j)$ satisfies (GCC).
\item If $\gamma(t) \in V_j \setminus \Omega_j$, then $\gamma(t) \in \omega$ since by construction $V_j \setminus \Omega_j \subset \omega$.
\item If $\gamma(t) \in M \setminus V$, then $\gamma(t) \in \omega$ since from Theorem \ref{thm-geometry} we have $\omega \supset (M \setminus V)$.
\end{itemize}
This completes the proof.
\end{proof}

\begin{Remark} \rm 
In Theorem \ref{thm-result2}, it possible that $M \setminus V$ contains closed (trapping) generalized geodesics. But in this case Theorem \ref{thm-geometry} 
guarantees that it remains inside $\omega$. As a consequence, we have proved that all $\omega \in [\omega_{\varepsilon}]$ 
satisfies (GCC), but we do not know if $(\omega \cap \partial M, T_0)$ satisfies the {\it geodesic condition} for some $T_0 >0$. \qed
\end{Remark}

\subsection{Decomposition in overlapping sets} 

As mentioned before, our construction seeks fulfill the assumptions 
of an observability result in \cite{ty2002}. In a first approach, 
it is required that function $d$ 
has no critical points in $M$. Note that our Theorem \ref{thm-geometry} grants only $d$ has no critical points in $V$. 
Nevertheless, this restriction can be weakened to a framework of overlapping sub-domains. 

\begin{Definition}
We say that $M$ admits a family of overlapping sub-domains $\{\Omega_j\}_{j=1}^k$ if 
\begin{itemize}
\item $M = \bigcup^k_{j=1} \Omega_j$,
\item for each $j \in \{1, \ldots, k\}$, there exists at least one $i \in \{1,\ldots, k\}\setminus \{j\}$ such that $\Omega_j \cap \Omega_i \ne \emptyset$.
\end{itemize}
\end{Definition}

\begin{Theorem} \label{thm-overlapp} 
Given $\varepsilon>0$ and $\omega \subset [\omega_{\varepsilon}]$, let 
$V_j,d_j$, $j=1, ... , k$, given by Theorem \ref{thm-geometry}. Then there exists a finite collection overlapping sub-domains $\{\Omega_j\}^k_{j=1}$ of $M$ such that
\begin{enumerate}
\item $V_j \subset \Omega_j$ for all $j=1,...,k$,
\item $\Omega_j \cap \omega \neq \emptyset$ for all $j=1,...,k$.
\end{enumerate}
Moreover, there exist functions $d_j : M \to \mathbb{R}$, $j=1, \cdots ,k$, such that
\begin{enumerate}
\item[3.] $d_{j} \in C^\infty(M), $ 
\item[4.] $\nabla^2d_j(X,X)\ge |X|_{\bf g}, \quad \forall \, X \in T_q M, \;\;  \forall \, q \in \Omega_j$,   
\item[5.] $\inf _{\Omega_j} |\nabla d_{j}| >0, \ \ \ \inf_{q \in \Omega_j}d_j(q)>0$,
\item[6.] $\langle \nabla d_j(x),n(x)\rangle <0$ on $\partial M \cap \overline{V_j}$.
\end{enumerate}
\end{Theorem}

\begin{proof} 
We recall that in the proof of Theorem \ref{thm-geometry}, 
there exists $k \in \mathbb{N}$ such that $V = \bigcup^k_{j=1} V_j$ where each 
$V_j \subset \widetilde{W_j}$, $j=1,...,k$, satisfies ({\it 1.})-({\it 2.}) 
with $\Omega_j = \widetilde{W_j}$. In addition, there exists a family 
of functions $\widetilde{d}_j:\widetilde{W_j}\to \mathbb{R}$ such that 
$$  
\widetilde{d}_j = 
\left\{ 
\begin{array}{l} 
(\ref{ddd}) \;\; \textrm{if} \;\; \widetilde{W_j} \cap \partial M \neq \emptyset, \\
(\ref{dddd}) \;\; \textrm{if} \;\; \widetilde{W_j}\cap \partial M = \emptyset,
\end{array} 
\right.
$$
and satisfies ({\it 4.})-({\it 5.}).

\medskip

On the other hand, if $\widetilde{W}_j \cap \partial M \neq \emptyset$, then as in Step 2 of the proof of Theorem \ref{thm-geometry}, there exists $\widetilde V'_j \subset \widetilde{M}$ 
such that $\widetilde{W}_j \subset \widetilde V'_j \cap M$. 
Then, there exists $W_j\subset \widetilde V'_j \cap M$ such that 
$\widetilde{W}_j \subset W_j$ and this allow us define 
$d_j:W_j \to \mathbb{R}$ as in (\ref{ddd}) so that ({\it 4.})-({\it 5.}) hold. 
Moreover one has 
$$
\langle \nabla {d_j}(q),n(q) \rangle < 0, \;\; q \in \widetilde{W}_j.
$$
It is worthy noting that above sign condition does not hold for all $q \in W_j$.

\medskip

Then, defining $\Omega_j=\widetilde{W_j}$ with $\widetilde{W_j}=W_j$ when $\widetilde{W_j}\cap \partial M \neq \emptyset$ and 
$$
	d_j=\left\{\begin{array}{l}
	d_j \ \ \  \textrm{ if } \ \Omega_j\cap \partial M \neq \emptyset, \\
	\widetilde{d}_j \ \ \  \textrm{ if } \ \Omega_j\cap \partial M = \emptyset,
	\end{array}\right.
$$
we see that ({\it 1.})-({\it 2.}) and ({\it 4.})-({\it 6.}) hold. 
Moreover $d_j \in C^\infty(\Omega_j)$. Finally, to prove  ({\it 5.}) 
it is enough to take unit partition over each $\widetilde{d}_j$. This concludes the proof.
\end{proof}

\section{Observability and unique continuation} 
\label{sec-obs}
\setcounter{equation}{0}

The objective of this section is to establish a new observability inequality, in terms of potential energy, 
for a large class of linear wave equations within our framework of admissible $\varepsilon$-controllable regions. 
We consider wave equations with potentials of the form 
\begin{equation} \label{L1} 
\left\{
\begin{aligned}
& \partial_t^2 w - \Delta w = p_{0}w+p_1 \partial_t w \quad \mbox{in} \;\; M \times (0,T) , \\
& w = 0  \quad  \mbox{on} \;\; \partial M \times (0,T),  
\end{aligned}
\right.
\end{equation}
where    
\begin{equation} \label{p0p1-1}
p_0\in L^2(0,T;L^2(M)) \;\; \mbox{and} \;\; p_1 \in L^{\infty}(0,T;L^{\infty}(M)),      
\end{equation} 
and for any weak solution $w \in L^2(0,T;H^1_0(M))\cap H^1(0,T;L^2(M))$, 
there exists a constant $C_T>0$ such that 
\begin{equation} \label{p0p1-2}
\| p_0 w \|_{L^2(0,T;L^2(M))} \le C_T \| w \|_{L^2(0,T;H^1_0(M))}. 
\end{equation}
Recall the notation already given in the Introduction for the natural phase space for the problem, $\mathcal{H}=H^1_0(M)\times L^2(M)$, equipped with norm $\|(u,v)\|^2_{\mathcal{H}}=\|\nabla u\|^2_2+\|v\|^2_2$.
As noticed before, we shall revisit a result by Triggiani and Yao \cite{ty2002}. 

\subsection{A result by Triggiani and Yao revisited}

\begin{Theorem}{\rm \cite[Theorem 10.1.1]{ty2002} } \label{thm-TYOVER}
	Let $(M, {\bf g})$ be an $N$-dimensional connected compact Riemannian manifold  of class $C^{\infty}$ with smooth boundary $\partial M$. 
	\begin{enumerate}
		\item Assume there exists a finite collection of overlapping sub-domains $\{\Omega_{j} \}_{j \in J}$ such that  
		for each $\Omega_j$, there exists a function $d_j: M \to \mathbb {R}$ satisfying:
		\begin{itemize}
			\item $d_{j} \in C^{\infty}(M)$ and $\displaystyle \inf_{\Omega_j} d_j >0$,  
			\item $\nabla^2d_j(X,X)\ge |X|_{\bf g}$, $\forall \, X \in T_q M$, $q \in \Omega_j$,   
			\item $\displaystyle \inf_{{\Omega}_j} |\nabla d_{j}| >0$.
		\end{itemize}
		\item Define the boundary regions  
		\begin{equation} \label{YT-Gamma}
		\Gamma_0 =\bigcup_{j \in J} \left\{ x \in \partial M \, | \, \langle \nabla d_j(x),\nu(x) \rangle \le 0 \right\}  
		\quad \mbox{and} \quad \Gamma_1 = \partial M \setminus \Gamma_0.
		\end{equation}
	\end{enumerate} 
	Then, for any solution $w$ of \eqref{L1} with $p_0,p_1$ satisfying \eqref{p0p1-1}-\eqref{p0p1-2}  
	and $T>0$ sufficiently large, there exists a constant $k_T>0$ such that 
	\begin{equation} \label{YT-Obs}
	\int_0^T \int_{\Gamma_1} \left( \frac{\partial w}{\partial \nu} \right)^2 d\Gamma_1 \, dt \ge 
	k_T \left( \|(w(0),\partial_t w(0))\|^2_{\mathcal{H}} + \|(w(T),\partial_t w(T))\|^2_{\mathcal{H}} \right).
	\end{equation}
	In addition, 
	\begin{equation} \label{YT-ucp}
	\text{if} \ \ \left. \frac{\partial w}{\partial \nu} \, \right|_{\Gamma_1 \times (0,T]} = 0 
	\quad \mbox{then} \quad w=0 \quad \mbox{in} \quad M \times [0,\infty).
	\end{equation}
\end{Theorem}

\medskip

\begin{Remark} \rm In \cite[Theorem 10.1.1]{ty2002} the above result is presented with two overlapping sub-domains and assumes that $p_0 \in L^\infty(M\times [0,T])$. The generalization to a finite number of overlapping sub-domains is a standard task. On the other hand, after a careful revision of its proof, one may check that it can be extended for $p_0$ satisfying \eqref{p0p1-1}. Indeed, as noticed in \cite[Remark 1.1]{ty2002}, this was earlier observed in \cite[Remark 1.1.1]{LTZ} for the Euclidean framework. \qed
\end{Remark}

\begin{Remark} \rm The observability inequality \eqref{YT-Obs} is stated above with respect to 
the boundary $\Gamma_1$. We shall present a new observability inequality with respect to 
an admissible $\varepsilon$-controllable region $\omega$. To this end, we need a technical result, in a context of Riemannian manifolds, which allows carrying area integrals over volume integrals. This is presented in the next section.  \qed
\end{Remark}

\subsection{A coarea formula} 

We begin with a known coarea formula for $N$-dimensional $C^{\infty}$ manifolds, here denoted by $(W,{\bf g})$.    
Accordingly, given a $C^{\infty}$ function $\phi: W \to \mathbb{R}$ and $f \in L^1(W)$, one has
\begin{equation} \label{coarea1}
\int_{M}|\nabla \phi|f \, dV_{\bf g} = \int_{\mathbb{R}}\int_{\Gamma(t)}f \, dV_{{\bf g}|_{\Gamma(t)}}dt,
\end{equation}
where $\Gamma(t):=\phi^{-1}({t})=\{p \in W \ | \ \phi(p)=t \}$ and $dV_{{\bf g}|_{\Gamma(t)}}$ is the induced measure on $\Gamma(t)$.
A proof of this result can be found in, e.g., Chavel \cite[Corollary I.3.1]{cha}. 
To our purpose, we prove a coarea relation involving $\Gamma_1$ appearing in \eqref{YT-Gamma}-\eqref{YT-Obs} 
inside a context of overlapping $\varepsilon$-controllable sets. 

\begin{Lemma} [Coarea relation] \label{lemma-AV}
Given $\varepsilon >0$, in the context of Theorem \ref{thm-overlapp}, let us define 
$$ 
\widehat{\Gamma_1} = \bigcup^k_{j =1} \{ x \in \partial M \;|\; \langle \nabla d_{j}, \nu \rangle > 0 \}. 
$$
Then there exists a constant $C = C(\varepsilon)>0$ such that 
\begin{equation} \label{AV-1}
\int_{\widehat{\Gamma_1}} f dV_{{\bf g}|_{\partial M}} \le C \int_{\omega} f dV_{\bf g},
\end{equation}
for any admissible $\varepsilon$-controllable set $\omega$ and 
\begin{equation} \label{ex}
f \in L^1(M) \;\; \mbox{with} \;\; f \ge 0 \;\; \mbox{a.e.}
\end{equation}
\end{Lemma}

\begin{proof} The proof will be given in three steps. 

\medskip 
	
{\it Step 1.} Let $\omega$ be an admissible $\varepsilon$-controllable set, that is, $\omega \in [\omega_{\varepsilon}]$. 
Then, in the context of Theorem \ref{thm-geometry} and \eqref{Vd}, there exists $ \varepsilon_0 \in (0, \varepsilon)$ and $V \subset M$ 
such that 
$$
\overline{M \setminus V} \in \chi_{\varepsilon_0}(M), \;\; \omega \cap V \in \chi_{\varepsilon-\varepsilon_0}(M).
$$ 
	Moreover, by constructing $ \omega \in M $ and by the compactness of $ M $, we have a finite number of connected components of $ \omega $ that intersect $ \partial M $, that is, there exists a number $ l \in \mathbb {N} $ such that
	$$
	\omega \cap \partial M = \bigcup^l_{j=1} \Gamma^j_{\omega}, 
	$$ 
	where $ \Gamma^j_{\omega}$ is the $j$-th connected component of $\omega \cap \partial M$.
	Note that $\widehat{\Gamma_1} \subset \omega$ (see Theorem \ref{thm-overlapp}), that is
$$
{\rm meas}_{\partial M}\widehat{\Gamma_1} < \varepsilon_0.
$$
Given the $\widehat{\Gamma_1} \subset \omega \cap \partial M$, then we will denote by $\Gamma^j_1$ to the connected components 
of $\widehat{\Gamma_1}$, given by
$$
\Gamma^j_1=\widehat{\Gamma_1} \cap \Gamma^j_{\omega},
$$
such that ${\rm meas}_{\partial M}\Gamma^j_1 < \varepsilon_0$ for each $j=1,2,...,l$.
	
\medskip
	
{\it Step 2.} Given a constant $h> 0$ and a set $A \subset \mathbb{R}^{N-1}$ we define 
$$
\mathrm{P}_{h}(A)=\{ (x_1,\dots,x_N) \in \mathbb{R}^N \ | \ (x_1,\dots,x_{N-1}) \in {\rm int}(A), \;\; 0 \le x_N<h \},
$$ 
called {\it open prism of base $A$ and height $h$}, which will play a fundamental role below. 

Let us fix $j$. Consider $\varepsilon^j_1 \in (0,\varepsilon)$ small enough, such that there is a $p \in \Gamma^j_1$ and a chart related to that point $(U_j,\phi_j=(x_1,\dots,x_N))$ with $\phi_j(p)=(0,\dots,0)$ such that the connected component $\Gamma^j_{\omega}$ is totally within that chart. Thus, there is a constant $h>0$ small enough, such that $\mathrm{P}_h(\Gamma^j_1)\subset \Gamma^j_{\omega} \subset \subset \phi_j(U_j)$ of $M$ and $\Gamma^j_{1}=\partial \phi^{-1}_j(P_h(\Gamma^j_1))$ (see Figure \ref{coarea}).

Observe that $(\phi_j^{-1}(\mathrm{P}_h(\Gamma^j_1)),{\bf g}|_{\phi_j^{-1}(\mathrm{P}_h(\Gamma^j_1))})$ is a smooth Riemannian submanifold of $M$ with boundary and dimension $N$, with the induced metric of $M$. In particular, we have
	
\begin{itemize}
\item The Lebesgue $\sigma$-algebra associated with $(\phi_j^{-1}(\mathrm{P}_h(\Gamma^j_1)),{\bf g}|_{\phi_j^{-1}(\mathrm{P}_h(\Gamma^j_1))})$ is the Lebesgue $\sigma$-algebra associated to $(M,{\bf g})$ intersected with $ \phi_j^{-1}(\mathrm{P}_h(\Gamma^j_1))$. 

\item If $B \subset \phi_j^{-1}(\mathrm{P}_h(\Gamma^j_1)) \subset \omega$ is a measurable set in $(M,{\bf g})$, then it is measurable in $(\omega,{\bf g}|_{\omega})$ and $(\phi_j^{-1}(\mathrm{P}_h(\Gamma^j_1)),{\bf g}|_{\phi_j^{-1}(\mathrm{P}_h(\Gamma^j_1))})$. 
Moreover,
$$
{\rm meas}_{M}(B)={\rm meas}_{\omega}(B)={\rm meas}_{\phi_j^{-1}(\mathrm{P}_h(\Gamma^j_1))}(B).
$$
since the three manifolds have the same dimension. Let us also observe that $(\omega,{\bf g}|_{\omega})$ and 
$(\phi_j^{-1}(\mathrm{P}_h(\Gamma^j_1)),{\bf g}|_{\phi_j^{-1}(\mathrm{P}_h(\Gamma^j_1))})$ have the metric of $M$ restricted to each of the manifolds.
\end{itemize}
	\begin{figure}[htb] 
	\begin{center}
		\includegraphics[scale=0.4]{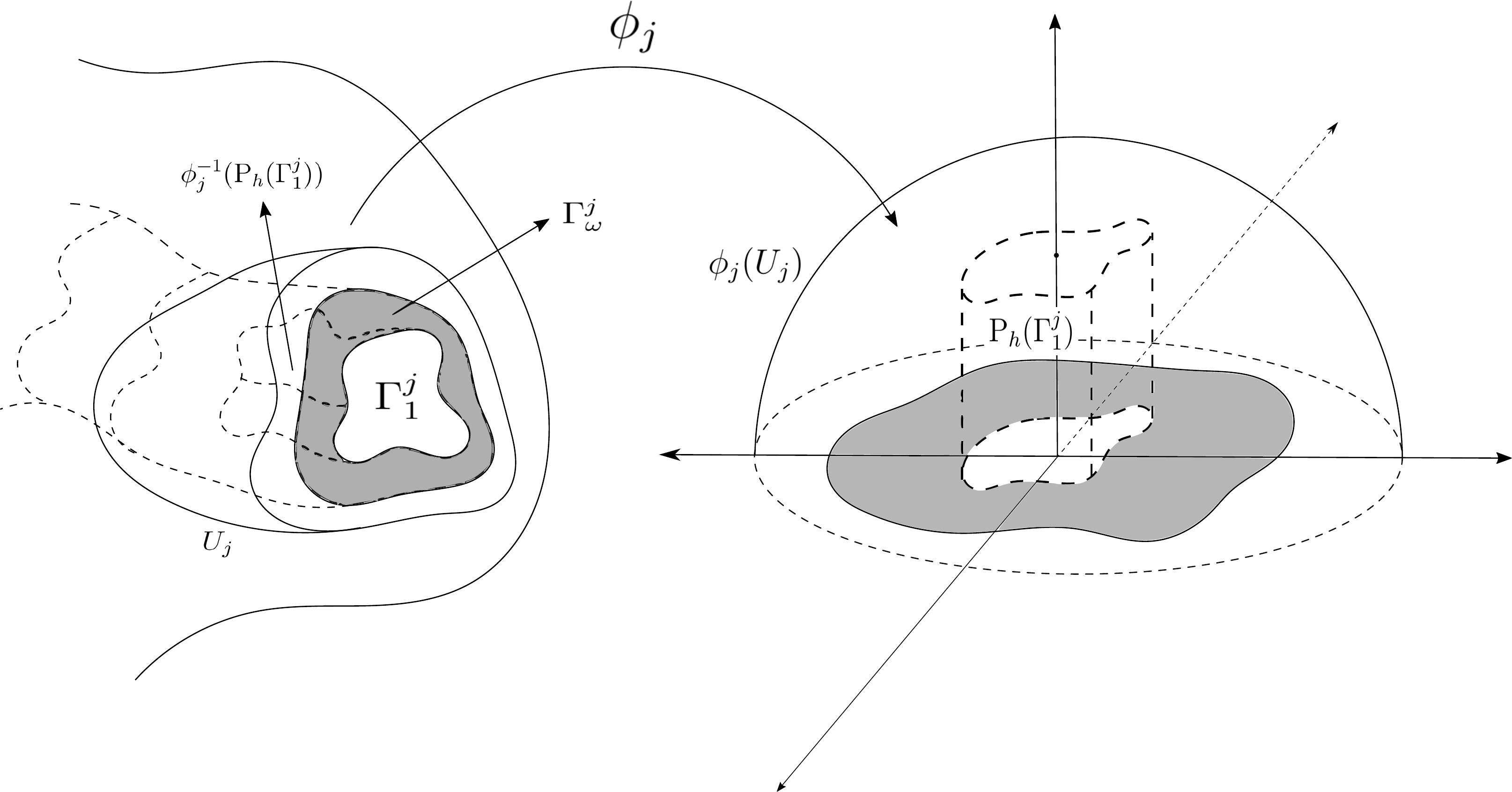} 
	\end{center}
	\caption{Note that $h>0$ depends directly on $\varepsilon_0>0$, thus, the open set $\phi_j^{-1}(\mathrm{P}_h(\Gamma^j_1))\subset U_j$  depends on the class $[\omega_{\varepsilon}]$.} \label{coarea}
\end{figure}

{\it Step 3.} Let us consider the $N$-dimensional manifold with boundary $(\phi_j^{-1}(\mathrm{P}_h(\Gamma^j_1)),{\bf g}|_{\phi_j^{-1}(\mathrm{P}_h(\Gamma^j_1))})$  and the map in $C^\infty(\phi_j^{-1}(\mathrm{P}_h(\Gamma^j_1)))$ given by $x_N:\phi_j^{-1}(\mathrm{P}_h(\Gamma^j_1)) \to \mathbb{R}$, 
where 
\begin{equation} \label{x3}
\Gamma^j_1 =x_N^{-1}(0), \ \ {\rm Img}(x_N)=[0,h), \ \ 0<|\nabla x_N|<C_{\bf g},
\end{equation}	
for a constant $C_{\bf g}>0$ that depends on the metric ${\bf g}$.
	
Thus, setting $(W,{\bf g})=(\phi_j^{-1}(\mathrm{P}_h(\Gamma^j_1)),{\bf g}|_{\phi_j^{-1}(\mathrm{P}_h(\Gamma^j_1))}), \Gamma^j_1=\Gamma(0)$ and  $\phi=x_N$, 
by the coarea formula \eqref{coarea1}, 
	$$
	\int_{\phi_j^{-1}(\mathrm{P}_h(\Gamma^j_1))}^{ }|\nabla x_N|fdV_{{\bf g}|_{\phi_j^{-1}(\mathrm{P}_h(\Gamma^j_1))}}=\int_{\mathbb{R}}^{ }\int_{\Gamma(t)}^{ }fdV_{{\bf g}|_{\partial \phi_j^{-1}(\mathrm{P}_h(\Gamma^j_1))}}dt,
	$$ 
	for all $f:M \to \mathbb{R}$ in $L^1(\phi_j^{-1}(\mathrm{P}_h(\Gamma^j_1)))$.
	
Then, taking $f$ satisfying (\ref{ex}) and by (\ref{x3}), 
we obtain
	\begin{align*}
	\int_{\Gamma^j_1}^{ }fdV_{{\bf g}|_{\partial M}}&\le  \int_{\Gamma(0)}^{ }fdV_{{\bf g}|_{\partial\phi_j^{-1}(\mathrm{P}_h(\Gamma^j_1))}} \\
	&\le  \int_{\mathbb{R}}^{ }\int_{\Gamma(t)}^{ }{f}dV_{{\bf g}|_{\partial\phi_j^{-1}(\mathrm{P}_h(\Gamma^j_1))}}dt \\
	&\le  C_{\bf g}  \int_{\phi_j^{-1}(\mathrm{P}_h(\Gamma^j_1))}^{ }fdV_{{\bf g}|_{\phi_j^{-1}(\mathrm{P}_h(\Gamma^j_1))}} \\
	&\le C_{\bf g} \int_{\omega}f dV_{{\bf g}}.
	\end{align*}
Finally, repeating the process for each $j=1,2,...,l$ and since $\bigcup^l_{j=1} \Gamma^j_1=\widehat{\Gamma_1}$, the proof is complete.
\end{proof}

\subsection{New observability and unique continuation}

Now we are ready to establish our observability inequality that is stated with potential energy instead of the usual kinetic energy. Moreover it is specially designed for using measure-controlled damping regions. 

\begin{Theorem} \label{thm-ucp} Let $(M,{\bf g})$ be an $N$-dimensional connected compact Riemannian manifold of class $C^{\infty}$ with smooth boundary and let $w \in L^2(0,T;H^1_0(M))\cap H^1(0,T;L^2(M))$ 
	be a solution of the linear problem \eqref{L1} with $p_0,p_1$ satisfying \eqref{p0p1-1}-\eqref{p0p1-2}. 
	Then, for any admissible $\varepsilon$-controllable region $\omega \subset M$ we have: 
	\begin{enumerate}
		\item Observability: for $T>0$ sufficiently large, there exists a constant $k_T>0$ 
		such that 
		\begin{equation} \label{obsWAVE} 
		\int_0^T \int_{\omega} |\nabla w|^2 dxdt \ge 
		k_T \left( \|(w(0),\partial_t w(0))\|^2_{\mathcal{H}} + \|(w(T),\partial_t w(T)) \|^2_{\mathcal{H}} \right).  
		\end{equation}
		\item Unique continuation: for the above $T>0$, if $w=0$ in $\omega \times (0,T)$ then $w=0$  in $M \times [0,\infty)$. 
	\end{enumerate} 
\end{Theorem}

\bigskip

\begin{proof}
Fix $\varepsilon>0$, by Theorem \ref{thm-overlapp} and Theorem \ref{thm-TYOVER} with $J =\{1,...,k\}$, there exists $k_T>0$ depending on $\varepsilon,T$ and $C_T$ such that 
\begin{equation*}
	\int_0^T \int_{\Gamma_1} \left( \frac{\partial w}{\partial n} \right)^2 d\Gamma_1 \, dt \ge 
	k_T \big( \|(w(0),\partial_t w(0))\|^2_{\mathcal{H}} + \|(w(T),\partial_t w(T))\|^2_{\mathcal{H}} \big), 
\end{equation*}
where
$$
\Gamma_1 =\bigcup^k_{j =1} \left\{ x\in \partial M \ |  \, \langle \nabla d_j(x), \nu(x) \rangle > 0  \right\} \subset \omega.
$$
Then, since $|\langle \nabla w, \nu \rangle| \le |\nabla w|$, we can apply coarea relation \eqref{AV-1} with $f=|\nabla w|^2$ and $\overline{\Gamma_1} = \Gamma_1$. This shows \eqref{obsWAVE}. Finally, if $w=0$ in $\omega \times (0,T]$, then \eqref{YT-ucp} implies promptly $w=0$ on $M \times [0,\infty)$.
\end{proof}

\section{Dynamics of semilinear wave equations} 
\label{sec-dynamics}
\setcounter{equation}{0}

This section is devoted to establish the existence of global attractors for dynamics of wave equations featuring locally distributed damping 
on admissible $\varepsilon$-controllable regions and nonlinear forcing terms with critical Sobolev growth. 

\medskip

\subsection{Assumptions and results}  \label{sec-critical}

Let $( M, {\bf g})$ be a $3$-dimensional connected compact Riemannian manifold of class $C^{\infty}$, with smooth boundary $\partial M$. We are concerned with the semilinear wave equation 
\begin{equation} \label{NW}
\left\{
\begin{aligned}
& \partial_{t}^2u-\Delta u + a(x)g(\partial_t u) + f(u) = 0 \quad \mbox{in} \quad M \times (0,\infty), \\
& u = 0 \quad \mbox{on} \quad \partial M \times (0,\infty), \\
& u(x,0) = u_{0}(x), \quad \partial_{t}u(x,0) = u_{1}(x), \quad x \in M. 
\end{aligned}
\right.
\end{equation}
We assume that    
\begin{equation}  \label{f1}
f \in C^1(\mathbb{R}), \quad f(0)=0, \quad |f'(z)|\le C_{f}(1+|z|^2), \quad \forall z \in \mathbb{R}, 
\end{equation}
for some constant $C_{f}>0$, 
\begin{equation} \label{f2}
\liminf_{|z| \to \infty} \frac{f(z)}{z} > -\lambda_1,
\end{equation}
where $\lambda_1 > 0$ is the first eigenvalue of $-\Delta$ in $M$ with homogeneous Dirichlet boundary condition, and 
\begin{equation} \label{g}
g \in C^1(\mathbb{R}), \quad g(0) = 0, \quad m_1 \le g'(z) \le m_2, \quad \forall z \in \mathbb{R}, 
\end{equation}
for some constants $m_1,m_2 > 0$. For the damping coefficient, there exists some $a_0 >0$, 
\begin{equation} \label{damping}
a \in L^\infty(M), \quad a \ge a_0 \;\; \mbox{a.e. in} \;\; \omega, 
\end{equation}
where $\omega$ is a suitable open set of $M$.

As we will see in Theorem \ref{thm-wp}, under above assumptions, problem (\ref{NW}) is well-posed in $\mathcal{H}$.
Then its solution operator defines a nonlinear $C^0$ semigroup $\{S(t)\}_{t \ge 0}$ on $\mathcal{H}$. 
Often the corresponding continuous dynamical system generated by the problem \eqref{NW} is denoted by $(\mathcal{H},S(t))$.  

\begin{Remark} \rm We recall that a global attractor for a dynamical system $(\mathcal{H},S(t))$  
is a compact set $\mathcal{A} \subset \mathcal{H}$ that is fully invariant and attracts bounded sets of $\mathcal{H}$. 
Also, given a compact set $K \subset \mathcal{H}$ its fractal dimension is defined by 
$$
\dim_F K = \limsup_{\epsilon \to 0} \frac{\ln(n_{\epsilon})}{\ln(1/\epsilon)},  
$$
where $n_{\epsilon}$ is the minimal number of closed balls of radius $\epsilon$ necessary to cover $K$.
See e.g. \cite{babin-vishik,hale,lady,temam} or \cite[Chapter 7]{CL-yellow}.
\end{Remark}

\begin{Theorem} [Attractors] \label{thm-attractor} 
Under assumptions \eqref{f1}-\eqref{g}, given $\varepsilon >0$, 
assume that \eqref{damping} is satisfied for some admissible $\varepsilon$-controllable set $\omega \subset M$. 
Then the dynamics of problem $(\ref{NW})$ has a global attractor $\mathcal{A}$ with finite fractal dimension and regularity $H^2(M) \times H^1(M)$.
\end{Theorem}

The existence of global attractors for wave equations with critical Sobolev exponent $p=3$ on bounded domains of $\mathbb{R}^3$ was firstly proved by Arrieta, Carvalho and Hale \cite{ach}, with a weak frictional damping defined over all the domain. Subsequently, 
Feireisl and Zuazua \cite{fz} proved the existence global attractors in the case of locally distributed damping, 
satisfying a geometric control condition. Their arguments used a unique continuation property by Ruiz \cite{ruiz}. 
In that direction, further properties like finite fractal dimension and regularity of attractors were achieved years later by Chueshov, Lasiecka and Toundykov \cite{clt}. In Theorem \ref{thm-attractor}, we consider the existence of a regular finite dimensional global attractor with  
a sharp measure-controlled damping region, that is, the damping region is any $\varepsilon$-controllable set. 
Our proof relies on the observability and unique continuation Theorem \ref{thm-ucp}. 
In addition, we only assume $f\in C^1$ instead $f \in C^2$ as in \cite{ach,clt,fz}. 

\begin{Remark} \rm 
The proof of Theorem \ref{thm-attractor} is divided into three parts. Firstly, we show that our system is gradient by using the unique continuation property in Theorem \ref{thm-ucp}. Then we apply a recent theory of quasi-stable systems (\cite{chueshov-book,CL-yellow}) and the observability inequality in Theorem \ref{thm-ucp} to prove asymptotic (compactness) smoothness of the system. Finally, by applying a classical existence result (e.g. \cite[Corollary 7.5.7]{CL-yellow}) we obtain a global attractor characterized by $\mathcal{A}=\mathbb{M}^{u}(\mathcal{N})$, the unstable manifold of the set $\mathcal{N}$ of stationary solutions of $(\ref{NW})$. \qed
\end{Remark}

\subsection{Well-posedness and energy estimates} \label{sec-wp} 

Let us write 
$$ 
U = \begin{bmatrix} u \\ \partial_t u \end{bmatrix}, \quad 
\mathbb{A} = \begin{bmatrix} 0 & -I \\ -\Delta & a(x)g(\cdot) \end{bmatrix}, \quad 
\mathbb{F} = \begin{bmatrix} 0 & 0 \\ f(\cdot) & 0 \end{bmatrix}.
$$
Then problem (\ref{NW}) is equivalent to the Cauchy problem 
$$ 
\partial_t U+\mathbb{A} U+\mathbb{F} U = 0, \quad U(0)=  \begin{bmatrix} u_0 \\ u_1 \end{bmatrix}, 
$$ 
defined in $\mathcal{H}$ with domain 
$$
D(\mathbb{A})= (H^2(M)\cap H^1_0(M)) \times H^1_0(M).
$$
From assumption \eqref{f1} it is well known that $\mathbb{F}$ is locally Lipschitz in $\mathcal{H}$ and then 
existence of weak and strong solutions follows from semigroup theory. 
The following existence result is essentially proved in \cite{clt,fz}. 

\begin{Theorem}[Well-possedness] \label{thm-wp} Assume that \eqref{f1}-\eqref{damping} hold. Then
	\begin{enumerate}
		\item For initial data $(u_0,u_1) \in \mathcal{H}$, problem \eqref{NW} possesses a unique weak solution 
		\begin{equation} \label{WPi}
		u \in C(\mathbb{R}^+;H^1_0(M))\cap C^1(\mathbb{R}^+;L^2(M)).
\end{equation} 
\item For initial data $(u_0,u_1) \in D(\mathbb{A})$, problem \eqref{NW} possesses a unique strong solution 
$$ 
u \in C(\mathbb{R}^+;H^2(M) \cap H^1_0(M)) \cap C^1(\mathbb{R}^+;H^1_0(M)).
$$ 
		\item Given $T>0$ and a bounded set $B$ of $\mathcal{H}$, there exists a constant $D_{BT}>0$ such that for any two initial values $z^i_0 \in B$, $i=1,2$, the corresponding solutions $z^i=(u^i,\partial_t u^i)$ satisfy
		\begin{equation} \label{uni1}
		\|z^1(t)-z^2(t)\|^2_{\mathcal{H}} \le D_{BT} \|z^1_0-z^2_0\|^2_{\mathcal{H}}, \quad \forall \, t \in [0,T],
		\end{equation} 
		where $D_{BT}>0$ is constant. 
	\end{enumerate} 
\end{Theorem}

The total energy of the problem \eqref{NW} is defined by    
\begin{equation} \label{def-energy}
\mathcal{E}(t)= \frac{1}{2} \|(u(t),\partial_t u(t)) \|_{\mathcal{H}}^2 
+ \int_{M} F(u(t)) \,dx ,
\end{equation}
with $F(u)=\int_{0}^{u}f(r)dr$. To avoid confusion, sometimes we write $\mathcal{E}_{u}$ instead $\mathcal{E}$.
We finish this section with some useful energy estimates.

\begin{Lemma} Under the assumptions of Theorem \ref{thm-wp}
\begin{enumerate}
\item The total energy is non-increasing and  
	\begin{equation} \label{energy2}
	\frac{d}{dt} \mathcal{E}_u(t) = - \int_{M} a(x)g(\partial_t u) \partial_t u \, dx, \;\; t \ge 0.
	\end{equation}
	
\item There exist constants $\beta, C_1,C_2>0$ such that 
	\begin{equation} \label{EE}
	\beta \Vert (u(t),\partial_t u(t)) \Vert^2_{\mathcal{H}}  - C_1  \le \mathcal{E}_{u}(t) 
	\le C_2 \big( 1+\Vert (u(t),\partial_t u(t)) \Vert_{\mathcal{H}}^4), \;\; t \ge 0.
	\end{equation}
\item There exists a constant $C_0>0$ such that for any initial data $(u_0,u_1) \in \mathcal{H}$ 
	\begin{equation} \label{compE}
 \|(u(t),\partial_t u(t))\|^2_{\mathcal{H}} \le C_0(1+\|(u_0,u_1)\|^4_{\mathcal{H}}), \;\; t \ge 0.
	\end{equation}	
\end{enumerate}
\end{Lemma}
	 
\begin{proof} By a density argument we can assume the solutions are regular. Then multiplying the equation in \eqref{NW} by $\partial_t u$ and integration by parts imply that \eqref{energy2} holds. 
To prove the first inequality of \eqref{EE} we observe that assumption \eqref{f2} implies that there exists 
$\beta,C > 0$ (depending of $f$) such that $F(u) \ge - \delta u^2 - C$ for all $u \in \mathbb{R}$. 
Then 
$$
\int_M F(u) \, dx \ge \frac{1}{2}\left(- 1 + \frac{\delta}{\lambda_1} \right) \| \nabla u \|_2^2 - C |M|, 
$$	
which implies \eqref{EE} with $\beta = \delta/(2 \lambda_1)$. The second inequality of \eqref{EE} follows from the growth condition of $f$. 
Finally, the proof of \eqref{compE} follows from \eqref{EE} and the fact that energy is non-increasing.  
\end{proof}

\subsection{Gradient structure} \label{sec-gradient} 

A dynamical system $(\mathcal{H},S(t))$ is gradient if it possesses a Lyapunov functional, that is,  
a function $\Psi: \mathcal{H} \to \mathbb{R}$ such that $t\mapsto \Psi(S(t)z)$ is non-increasing and if 
\begin{equation} \label{lyapunov}
\Psi(S(t)z)=\Psi(z), \quad \forall \, t \ge 0,
\end{equation}
then $z$ is fixed point of $S(t)$. 

\begin{Theorem}[Gradient structure] \label{thm-gradient} 
Under assumptions of Theorem \ref{thm-attractor} the dynamical system $(\mathcal{H},S(t))$ associated to the problem \eqref{NW} is gradient. Moreover, the total energy $\mathcal{E}(t)$ as a Lyapunov functional. 
\end{Theorem} 

\begin{proof} We show that total energy defined in \eqref{def-energy} is a Lyapunov functional. 
Consider a solution $(u,\partial_tu)=S(t)z$ of \eqref{NW}. Then by \eqref{energy2}
$$
\Psi(S(t)z) = \frac{1}{2}\| (u(t),\partial_t u(t)) \|^2_{\mathcal{H}}+\int_{M} F(u(t)) \, dx
$$
satisfies  
\begin{equation} \label{h1}
\frac{d}{dt} \Psi(S(t)z) = - \int_{M} a(x)g(\partial_t u(t))\partial_t u(t) \, dx, \quad t \ge 0. 
\end{equation}
This shows that $\Psi(S(t)z)$ is decreasing with respect to $t$. The rest of the proof is splitted into two steps.

\bigskip

{\it Step 1.} (The role of Lyapunov function) Suppose that $z_0=(u_0,u_1)$ satisfies \eqref{lyapunov}. Then \eqref{h1} implies that 
$$
\int_{M} a(x)g(\partial_t u)\partial_t u \, dx = 0,
$$
and from assumption \eqref{g} we infer that   
$$
\int_{\omega} |\partial_t u|^2 \, dx =0 \quad \mbox{and} \quad \int_{M} a(x) | g(\partial_t u ) |^2 \, dx = 0.
$$
Therefore $S(t)z_0=(u(t), \partial_t u(t))$ is a $C^0([0,T];\mathcal{H})$ solution of  
the undamped system 
\begin{equation} \label{G3}
\left\{
\begin{aligned}
& \partial^2_t u - \Delta u + f(u) = 0 \quad \hbox{in} \quad M \times (0,T), \\
& u = 0 \quad \hbox{on} \quad \partial M \times (0,T), \\
& u(0) = u_0, \quad \partial_t u(0)= u_1 \quad \hbox{in} \quad M,  
\end{aligned}
\right.
\end{equation} 
with supplementary condition 
\begin{equation} \label{over}
\partial_t u = 0 \quad \hbox{a.e. in} \; \omega \times (0,T).  
\end{equation} 
We see that if $(u_0^k,u_1^k) \in D(\mathbb{A})$ converge to $(u_0,u_1)$ in $\mathcal{H}$, 
then their corresponding strong solutions $(u^k(t),\partial_t u^k(t))$ satisfy 
\eqref{G3}-\eqref{over}. 

\bigskip

{\it Step 2.} (Applying the unique continuation property)   
Let us denote $w^k=\partial_t u^k$. Then we see that $(w^k, \partial_t w^k)\in C^0([0,T]; \mathcal{H})$ is 
a weak solution of 
$$ 
\left\{
\begin{aligned}
& \partial^2_t w^k - \Delta w^k + f'({u^k}){w^k}=0 \quad \mbox{in} \quad M\times (0,T), \\ 
& w^k = 0 \quad \hbox{on} \quad \partial M \times (0,T), \\
& w^k = 0 \quad \mbox{in} \quad  \omega \times (0,T).
\end{aligned}
\right.
$$  
For each $k \in \mathbb{N}$ we shall apply Theorem \ref{thm-ucp} with 
$$
p_0 = f'(u^k) \quad  \mbox{and} \quad p_1=0. 
$$
Clearly $p_1 \in L^{\infty}(0,T;L^{\infty}(M))$. Then it is enough to show that \eqref{p0p1-2} holds with $w^k$ replacing $w$. Indeed, we have 
\begin{align*}
\int_M | p_0 w^k |^2 \, dx 
& \le  C \int_M (1 + |u^k|^4)|w^k|^2 \, dx \\
& \le C \left( 1 +  \|u^k\|_{L^6(M)}^4 \right) \|w^k\|_{L^6(M)}^2. 
\end{align*}
Since $u^k \in C^0([0,T];H^1_0(M))$ and $w^k \in H^1_0(M)$, 
$$
\| p_0 w^k \|_{L^2(0,T;L^2(M))}^2 \le C_T^k \int_0^T \| \nabla w^k(t)\|_{2}^2 \,dt,   
$$
which is the required estimate. 
Applying Theorem \ref{thm-ucp}, we get $w^k=0$ in $M \times [0,T]$ for each $k \in \mathbb{N}$,  
so that $\partial_t u(t)=0$ a.e. in $M$, for all $t \in [0,T]$. Therefore 
$z_0 = (u_0,0)$ is a stationary solution. This concludes the proof.
\end{proof}

\subsection{Quasi-stability} 

In order to prove the asymptotic smoothness and further properties of global attractors,  
we apply a recent theory of quasi-stable systems \cite{chueshov-book,CL-yellow} 
that is very useful for studying long-time dynamics of nonlinear wave equations. 
Its framework is based on a system $(H,S(t))$ with $H = X \times Y$, where $X$ and $Y$ are Banach 
spaces and $X \hookrightarrow Y$ compactly. 
Moreover, given $z_0=(u_0,u_1) \in H$, the trajectory $S(t)z_0 = (u(t),\partial_t u(t))$ satisfies    
$$
u \in C^0(\mathbb{R}^{+};X) \cap C^1(\mathbb{R}^{+};Y).
$$ 
In order to present the definition of quasi-stability, given a set $B$ and $z^1,z^2 \in B$, 
let us denote the corresponding trajectories as
$$
S(t)z^i = (u^i(t),\partial_t u^i(t)), \;\; i=1,2, \;\; t \ge 0. 
$$
Under the above setting, the dynamical system $(H,S(t))$ is said to be quasi-stable in a set $B \subset H$ if there exist positive constants $\zeta$ and $C_B$ such that for any $z^1,z^2 \in B$,
\begin{equation} \label{quasi-stability} 
\|S(t)z^1 - S(t)z^2\|_{H}^2 \le e^{-\zeta t} \|z^1-z^2\|_{H}^2 + C_{B} \sup_{s \in  [0,t]} \| u^1(s) - u^2(s) \|_W^2,   
\end{equation}
where $W \subset Y$ is a Banach space with compact embedding $X \hookrightarrow W$. 

\begin{Remark} \label{rem-quasistability} \rm Quasi-stable systems have three major features with respect to global attractors. $(a)$ If a system is quasi-stable on any forward invariant bounded set, then it is asymptotically smooth (cf. \cite[Proposition 7.9.4]{CL-yellow}). 
$(b)$ If a system possesses a global attractor $\mathcal{A}$ and it is quasi-stable on $\mathcal{A}$, then that attractor has finite fractal dimension (cf. \cite[Theorem 7.9.6]{CL-yellow}). $(c)$ The constant $C_B>0$ in (\ref{quasi-stability}) can be replaced by a $L^1_{\rm loc}$ function. 
However, in the case $C_B$ is a constant, the complete trajectories $(u(t),\partial_t u(t))$ inside global attractor have further time-regularity, namely,  
\begin{equation} \label{t-regular}
\| \partial_t u (t) \|_{X} + \| \partial_t^2 u(t) \|_{Y} \le R , \quad t \in \mathbb{R}, 
\end{equation}
where $R>0$ is a constant (cf. \cite[Theorem 7.9.8]{CL-yellow}). \qed  
\end{Remark}

In the following we prove that our system is quasi-stable on any bounded forward invariant set. 

\begin{Theorem} [Quasi-stability] \label{thm-stable} The dynamical system $(\mathcal{H},S(t))$ generated by problem \eqref{NW} 
is quasi-stable on any forward invariant bounded set $B$ of $\mathcal{H}$. More precisely, there exist positive constants $\zeta$ and $C_B$ such that any two given solutions $z^i=(u^i,\partial_t u^i)$, $i=1,2$ of problem \eqref{NW} with initial data $z^1_0, z^2_0 \in B$, fulfills  
\begin{equation} \label{esta}
\| z^1(t) - z^2(t) \|_{\mathcal{H}}^2 \leq e^{-\zeta t} \|z^1_0 - z^2_0 \|_{\mathcal{H}}^2 
+ C_B\sup_{s \in [0,t]} \|u^1(s)-u^2(s)\|^2_{L^3(M)}, \;\; t\ge 0.   
\end{equation}
\end{Theorem}

Note that (\ref{esta}) is a quasi-stability inequality like (\ref{quasi-stability}) since 
$X=H^1_0(M)$ is compactly embedded in $W=L^3(M)$.

The proof of this theorem will be given through several lemmas. 
Firstly we see that solution operator $S(t)$ of problem \eqref{NW} defined on the phase space $\mathcal{H}$
satisfies \eqref{WPi} and consequently our system $(\mathcal{H},S(t))$ falls in the framework of quasi-stable systems. 
Therefore to prove the quasi-stability on forward invariant bounded sets of $\mathcal{H}$ it is enough to prove the 
inequality \eqref{esta}. To this end, putting $w = u^1 - u^2$, 
we see that $(w,\partial_t w)$ satisfies the equation
\begin{equation} \label{w1}
\left\{\begin{array}{l}
\partial^2_t w-\Delta w = p_0 w+p_1 \partial_t w \quad \mbox{in} \quad M \times (0,\infty), \\
w=0 \quad \mbox{on} \quad \partial M\times (0,\infty),\\
w(0) = w_0, \;\; \partial_t w(0) = w_1 \quad \mbox{in} \quad M,
\end{array}\right.
\end{equation}
where $(w_0,w_1)= z^1_0 - z^2_0$,  
\begin{equation} \label{p0p1-our}
p_0 = - f'(\alpha u^1 + (1-\alpha) u^2) \;\; \mbox{and} \;\; p_1 = - ag'(\beta \partial_t u^1 + (1-\beta) \partial_t u^2), 
\end{equation}
$\alpha, \beta \in [0,1]$. The energy of the system is defined by 
$$
E (t) = \frac{1}{2} \| (w(t),\partial_t w(t)) \|_{\mathcal{H}}^2 
= \frac{1}{2} \| z^1(t) - z^2(t) \|_{\mathcal{H}}^2. 
$$
We see that  
\begin{equation} \label{Eprime}
\frac{d}{dt} E \le -a_0m_1 \|\partial_t w\|_{L^2(\omega)}^2-\int_{M}p_0 w\partial_t w \, dx.
\end{equation}
In order to establish estimate \eqref{esta} we shall use perturbed energy method.  
Let us define 
$$
\phi(t)=\int_{M} w(t) \partial_t w(t)\,dx, \quad \psi(t) = \int_{\omega} w(t) \partial_t w(t) \, dx,
$$
and  
$$
\Phi(t)= \mu E(t) + \eta \phi(t) + \psi(t), 
$$
where $\mu,\eta>0$ are to be fixed later.

\medskip

\begin{Lemma} \label{lemma-Ew}	Under the above assumptions and notations,
\begin{enumerate}
	\item For $\mu$ large and $\eta \le 1$ we have 
	\begin{equation} \label{equivalent}
	\beta_1 E(t) \le \Phi(t) \le \beta_2 E(t), \quad t \, \ge 0, 
	\end{equation}
	with $\beta_1 = \mu - \frac{2}{\sqrt{\lambda_1}}$ and 
	$\beta_2 = \mu + \frac{2}{\sqrt{\lambda_1}}$. 
	\item There exists a constant $C>0 $ such that
	\begin{equation} \label{phi-prime}
	\frac{d \phi}{dt} \le -E -\frac{1}{2}\|\nabla w\|^2_{L^2(M)} 
	+ 2 \|\partial_t w\|^2_{L^2(M)}+C\|w\|^2_{L^3(M)} +\int_{M} p_0|w|^2 \, dx,
	\end{equation}
	\item There exists a constant $C>0 $ such that
$$
\frac{d \psi}{dt} \le -\|\nabla w\|^2_{L^2(\omega)}+2\|\partial_t  w\|^2_{L^2(\omega)} 
+ C\|w\|^2_{L^3(M)}+\int_{\omega} p_0|w|^2 \, dx.
$$
\end{enumerate}
\end{Lemma}	

\begin{proof} The proof is standard. Let us verify estimate \eqref{phi-prime}. Using (\ref{w1}),
	$$
	\frac{d \phi}{dt} = \int_{M} (\Delta w + p_0 w + p_1 \partial_t w) w \, dx + \| \partial_t w\|_{L^2(M)}^2. 
	$$
	From assumption (\ref{g}) we deduce
	$$
	\int_{M} p_1 (\partial_t w) w \, dx \le m_2 \| w\|_{L^2(M)} \|\partial_t w\|_{L^2(M)} \le 
	C_3 \|w\|^2_{L^3(M)} + \frac{1}{2} \| \partial_t w\|_{L^2(M)}^2.   
	$$
	Then  
	$$
	\frac{d \phi}{dt} \le - \|\nabla w\|_{L^2(M)}^2  + \frac{3}{2} \| \partial_t w\|_{L^2(M)}^2
	+ C_3 \|w\|^2_{L^3(M)} + \int_{M} p_0 | w|^2 \, dx,  
	$$
	which implies (\ref{phi-prime}). 
\end{proof}

From above lemma, \eqref{Eprime} and taking $\mu > 2/(a_0 m_1)$ it yields
\begin{align*}
\frac{d}{dt}\Phi(t)  \le - \eta E(t) + Z(t), \quad t\ge 0,
\end{align*}
where 
$$
Z = - \frac{3}{2} \|\nabla w\|^2_{L^2(\omega)} - \mu \int_{M} p_0w w_t \, dx 
+ 2 \eta \|w_t\|^2_{L^2(M)} + 2 \int_{M} |p_0 w^2| \, dx + C_{B} \| w \|_{L^3(M)}^2.
$$
Using \eqref{equivalent} and Gronwall lemma, we obtain  
\begin{equation} \label{withZ}
\Phi(t) \le e^{-\frac{\eta}{\beta_2}t} \Phi(0) 
+ \int_{0}^{t}e^{-\frac{\eta}{\beta_2}(t-s)} Z(s)\,ds.
\end{equation}
We shall estimate the integral term in \eqref{withZ} by applying the observability inequality \eqref{obsWAVE}.

\begin{Lemma} \label{lemma-p0w}	The functions $p_0,p_1$ defined in \eqref{p0p1-our} satisfy the assumptions \eqref{p0p1-1}-\eqref{p0p1-2} of Theorem \ref{thm-ucp}. In addition, there exists $C_{BT}>0$ such that 
	\begin{equation} \label{p0L1}
	\| p_0w \|_{L^1(0,T;L^2(M))} \le C_{BT} \sup_{t\in [0,T]} \| w(t) \|_{L^3(M)},
	\end{equation}
	for sufficiently large $T>0$. 
\end{Lemma}

\begin{proof} Clearly $p_1 \in L^{\infty}(0,T;L^{\infty}(M))$.
Now, from assumption \eqref{f1} there exists a constant $C>0$ such that     
	$$
	\int_{M} |p_0 w|^2 \, dx \le C\left( 1+ \| u^1 \|_6^4 + \| u^2 \|_6^4 \right) \| w \|_6^2.
	$$
Using \eqref{compE} and since $B$ is forward invariant, then $\| p_0 w \|^2_{L^2(M)} \le C_{BT} \| \nabla w \|_{L^2(M)}^2$. Integrating over $[0,T]$ we obtain \eqref{p0p1-2}. 

\medskip
	
	To prove \eqref{p0L1} we use Strichartz estimates. 
	Rewriting the wave equation in \eqref{NW} as 
	$$
	\partial_{t}^2u-\Delta u = G(x,t),
	$$
	with $G = - ag(\partial_t u)-f(u)$, we see that $G \in L^1(0,T;L^2(M))$.
Then we can apply Strichartz estimates \cite{stri} to a solution $u$ of \eqref{NW} with initial data $(u_0,u_1)$. Accordingly (see \cite{ginibre-velo,joly-laurent}), 
	for  
	$$
	\frac{1}{q} + \frac{3}{r} = \frac{1}{2}, \quad q \in [\mbox{$\frac{7}{2}$},\infty],
	$$
	we obtain, for some constant $C>0$,  
	\begin{equation} \label{strichartz}
	\| u \|_{L^q(0,T;L^r(M))} \le C \left( \| u_0 \|_{H^1_0(M)} + \| u_1 \|_{L^2(M)} 
	+ \| G \|_{L^1(0,T;L^2(M))} \right).
	\end{equation}
	In particular, for $q=4$ and $r=12$, we see that $u^1,u^2 \in L^{12}(M)$ and therefore   
	$$
	\|p_0 w\|_{L^2(M)}^2 \le C\left(1 + \| u^1 \|_{L^{12}(M)}^4 + \| u^2 \|_{L^{12}(M)}^4 \right) \| w \|_{L^3(M)}^2.
	$$
	Taking into account that \eqref{strichartz} is uniformly bounded for $z^1,z^2 \in B$, 
	\begin{align*}
	\int_0^T \|p_0 w\|_{L^2(M)} \, dt
	& \le C \int_0^T 
	\left(1 + \| u^1 \|_{L^{12}(M)}^2 + \| u^2 \|_{L^{12}(M)}^2 \right) \| w \|_{L^3(M)}\,dt \\ 
	& \le C \left(1 + \| u^1 \|^2_{L^4(0,T;L^{12}(M))} + \| u^2 \|^2_{L^4(0,T;L^{12}(M))} \right)\|w\|_{L^2(0,T;L^{3}(M))} \\ 
	& \le C_{BT} \sup_{t \in [0,T]} \| w(t) \|_{L^3(M)},
	\end{align*}
	which implies \eqref{p0L1}.  
\end{proof}  

\begin{Lemma} \label{lemma-M} For $T>0$ large we can choose $\eta \in (0,1)$ such that 
	\begin{equation} \label{M}
	\int_{0}^{T}e^{-\frac{\eta}{\beta_2}(T-s)} Z(s) \, ds \le C_{BT} \sup_{t \in [0,T]}\| w(t)\|_{3}^2,
	\end{equation}
	for some constant $C_{BT}>0$.
\end{Lemma}

\begin{proof} 
	We have seen that $p_0,p_1$ satisfy the assumptions of Theorem \ref{thm-ucp}. 
	Keeping in mind that $\beta_2 > 2 \lambda_1^{-\frac{1}{2}}$ and $\eta < 1$,  
	the observability inequality gives 
	\begin{align*}
	- \int_0^T e^{-\frac{\eta}{\beta_2}(T-s)} \| \nabla w(s) \|_{L^2(\omega)}^2 \, ds 
	& \le - e^{-\frac{\sqrt{\lambda_1}}{2}T} \int_0^T \| \nabla w(s) \|_{L^2(\omega)}^2 \, ds \\
	& \le - 2 k_T e^{-\frac{\sqrt{\lambda_1}}{2}T} E(0), 
	\end{align*}
	for $T>0$ large. Now from \eqref{p0L1} and \eqref{uni1}, given $\rho >0$, there exists $C_{BT\rho}>0$ such that 
	\begin{align*}
	\int_0^T e^{-\frac{\eta}{\beta_2}(T-s)} \int_{M} |p_0 w \partial_t w  | \,dx ds 
	& \le \| p_0w \|_{L^1(0,T;L^2(M))}  \| \partial_t w\|_{L^{\infty}(0,T;L^2(M))} \\ 
	& \le C_{BT\rho} \sup_{t \in [0,T]} \|w(t)\|_{L^3(M)}^2 + \rho E(0).
	\end{align*} 
	We also have   
	$$
	\int_0^T e^{-\frac{\eta}{\beta_2}(T-s)} \| \partial_t w \|_{2}^2 \, ds \le 2 T D_{BT} E(0),
	$$ 
	and  
	$$
	2\int_0^T e^{-\frac{\eta}{\beta_2}(T-s)} \int_{M} |p_0 w^2| \,dx ds \le C_{BT} \sup_{t \in [0,T]} \|w(t)\|_{L^3(M)}^2 .
	$$
	Combining the above estimates we obtain   
	$$
	\int_{0}^{T}e^{-\frac{\eta}{\beta_2}(T-s)} Z(s) \, ds 
	\le (4 \eta T D_{BT} + \mu \rho - 3 k_T e^{-\frac{\sqrt{\lambda_1}}{2}T} ) E(0) + C_{BT} \sup_{t \in [0,T]}\| w(t)\|_{L^3(M)}^2.
	$$
	Choosing 
	\begin{equation} \label{eta}
	\eta < \frac{k_T}{4D_{BT}}e^{-\frac{\sqrt{\lambda_1}}{2}T} \quad \mbox{and} \quad 
	\rho < \frac{k_T}{\mu}e^{-\frac{\sqrt{\lambda_1}}{2}T}  
	\end{equation} 
	\eqref{M} follows. 
\end{proof}

\paragraph{Proof of Theorem \ref{thm-stable} (conclusion).} Firstly we fix $T>0$ large according to the observability inequality. 
Then fix $\eta \in (0,1)$ satisfying \eqref{eta} and 
$$
\mu > \max \left\{ \frac{2}{a_0m_1} , \frac{2}{\sqrt{\lambda_1}} \right\}. 
$$
Then $\beta_1,\beta_2>0$ can be defined as in Lemma \ref{lemma-Ew}. Combining \eqref{withZ} and \eqref{M} we obtain 	
$$
\Phi(T) \le \gamma_{T} \Phi(0) + C_{BT} \sup_{t \in [0,T]} \| w(t) \|_{L^3(M)}^2, 
$$
where $\gamma_{T} = e^{-\frac{\eta T}{\beta_2}} < 1$. Since the system is autonomous, 
repeating the argument for $[T,2T]$ and so on (e.g. \cite[Lemma 8.5.5]{CL-yellow}), 
we obtain $\xi >0$ such that  
$$
\Phi(t) \le C_B e^{-\xi t} \Phi(0) + C_B \sup_{s \in [0,t]} \|w(s)\|^2_{L^3(M)}, \quad t \ge 0. 
$$
Finally, from \eqref{equivalent}, 
$$
E(t) \le \frac{\beta_2}{\beta_1}C_B e^{-\xi t} E(0) 
+ \frac{C_B}{\beta_1} \sup_{s \in [0,t]} \|w(s)\|^2_{L^3(M)}, \quad t \ge 0. 
$$
Therefore \eqref{esta} holds. \qed

\subsection{Proof of main result}

\paragraph{Proof of Theorem \ref{thm-attractor}.} 
$(a)$ We have proved that the system is asymptotically smooth and gradient. 
Also, we notice that $\Psi(z) \to \infty$ if and only if $\|z\|_{\mathcal{H}} \to \infty$. 
It remains to show that stationary solutions of \eqref{NW} are uniformly bounded. 
Indeed, if
$$
\| \nabla u \|_{2}^2 + \int_M f(u)u \, dx = 0, 
$$
using \eqref{f2} we can write     
$$
\int_{M} f(u)u \, dx \ge - \frac{\lambda_1}{4} \| u \|_2^2 - c_{f}, 
$$ 
for some constant $c_{f}>0$. This gives $\| \nabla u \|_{2}^2 \le 2c_{f}$, 
which shows that $\mathcal{N}$ is bounded. Then the existence of a global attractor $\mathcal{A}$  
follows from a classical result (e.g. \cite[Corollary 7.5.7]{CL-yellow}).
$(b)$ Theorem \ref{thm-stable} shows that our system is quasi-stable on the global attractor $\mathcal{A}$.  
Therefore, as mentioned in Remark \ref{rem-quasistability}, $\mathcal{A}$ has finite fractal dimension from \cite[Theorem 7.9.6]{CL-yellow}.
$(c)$ To see the regularity of attractor $\mathcal{A}$, we know from (\ref{t-regular}) that any complete trajectory 
$(u(t),\partial_t u(t))$ satisfies 
$$
\| \partial_t u (t) \|_{H^1_0(M)} + \| \partial_t^2 u(t) \|_{L^2(M)} \le R , \quad t \in \mathbb{R}.  
$$
Then, equation (\ref{NW}) gives $-\Delta u \in L^2(M)$ and therefore $(u,\partial_t u) \in (H^2(M) \cap H^1_0(M)) \times H^1_0(M)$. 
This completes the of Theorem \ref{thm-attractor} \qed

\paragraph{Acknowledgment.} M. M. Cavalcanti was partially supported by CNPq grant 300631/2003-0. 
T. F. Ma was partially supported by CNPq grant 312529/2018-0 and Fapesp grant 2019/11824-0.  
P. Mar\'{\i}n-Rubio was partially supported by 
Ministerio de Educaci\'on-DGPU grant PHB2010-0002-PC, Junta de Andaluc\'{\i}a grant P12-FQM-1492, 
and MINECO/FEDER grants MTM2015-63723-P and PGC2018-096540-B-I00. 
P. N. Seminario-Huertas was fully supported by CAPES/PROEX grant 8477445/D 
and CNPq grant 141602/2018-0.

\paragraph{Addresses}  

\begin{itemize} 
\item M. M. Cavalcanti

Departamento de Matem\'atica, Universidade Estadual de Maring\'a, 87020-900 Maring\'a, PR, Brazil (mmcavalcanti@uem.br)    

\item T. F. Ma and P. N. Seminario-Huertas

Instituto de Ci\^encias Matem\'aticas e de Computa{\c c}\~ao, Universidade de S\~ao Paulo, 13566-590 S\~ao Carlos, SP, Brazil (matofu@icmc.usp.br, pseminar@icmc.usp.br) 

\item P. Marin-Rubio 

Departamento Ecuaciones Diferenciales y An\'alisis Num\'erico, Universidad de Sevilla, c/ Tarfia s/n, 41012 Sevilla, Spain (pmr@us.es)
\end{itemize}


\begin{thebibliography}{40} 
\addcontentsline{toc}{section}{References}

\bibitem{ach} J. Arrieta, A. N. Carvalho and J. K. Hale, 
{\it A damped hyperbolic equation with critical exponent}, 
Comm. Partial Differential Equations 17 (1992) 841--866.

\bibitem{babin-vishik} A. V. Babin and M. I. Vishik, 
Attractors of Evolution Equations, 
Studies in Mathematics and its Applications 25, 
North-Holland, Amsterdam, 1992. 

\bibitem{blr} C. Bardos, G. Lebeau and J. Rauch, 
{\it Sharp sufficient conditions for the observation, control, and stabilization of waves from the boundary}, 
SIAM J. Control Optim. 30 (1992) 1024--1065. 

\bibitem{burq} N. Burq, 
{\it Contr\^olabilit\'e exacte des ondes dans des ouverts peu r\'eguliers}, 
Asymptot. Anal. 14 (1997) 157--191.

\bibitem{burq-gerard} N. Burq and P. G\'erard, 
{\it Condition n\'ecessaire et suffisante pour la contr\^olabilit\'e exacte des ondes}, C. R. Acad. Sci. Paris S\'er. I Math.  325 (1997), no. 7, 749--752.

\bibitem{cavalcanti1} M. M. Cavalcanti, V. N. Domingos Cavalcanti, R. Fukuoka and J. A. Soriano, 
{\it Uniform stabilization of the wave equation on compact surfaces and locally distributed damping -- a sharp result}, 
Trans. Amer. Math. Soc. 361 (2009) 4561--4580. 

\bibitem{cavalcanti2} M. M. Cavalcanti, V. N. Domingos Cavalcanti, R. Fukuoka and J. A. Soriano, 
{\it Asymptotic stability of the wave equation on compact manifolds and locally distributed damping: a sharp result}, 
Arch. Ration. Mech. Anal. 197 (2010) 925--964. 

\bibitem{doCarmo} M. P. do Carmo, 
Riemannian Geometry, Mathematics: Theory \& Applications, Birkh{\"a}user, Boston, 1992. 

\bibitem{cha} I. Chavel, Isoperimetric Inequalities: Differential Geometric and Analytic Perspectives, 
Cambridge University Press, 2001. 

\bibitem{chueshov-book} I. Chueshov, 
Dynamics of Quasi-Stable Dissipative Systems, Universitext, Springer, Cham, 2015. 

\bibitem{clt} I. Chueshov, I. Lasiecka and D. Toundykov, 
{\it Long-term dynamics of semilinear wave equation with nonlinear localized interior damping and a source term of critical exponent}, 
Discrete Contin. Dyn. Syst. 20 (2008) 459--509. 

\bibitem{CL-yellow} I. Chueshov and I. Lasiecka, 
Von Karman Evolution Equations. Well-Posedness and Long-Time Dynamics, 
Springer Monographs in Mathematics, Springer, New York, 2010. 

\bibitem{fz} E. Feireisl and E. Zuazua, 
{\it Global attractors for semilinear wave equations with locally distributed nonlinear damping and critical exponent}, 
Comm. Partial Differential Equations 18 (1993) 1539--1555. 

\bibitem{folland} G. Folland, Real Analysis: Modern Techniques and Their Applications, John Wiley \& Sons, Inc., New York, 1984. 	

\bibitem{ginibre-velo} J. Ginibre and G. Velo, 
{\it The global Cauchy problem for the nonlinear Klein-Gordon equation}, 
Math. Z. 189 (1985) 487--505. 

\bibitem{LH} L. H\"ormander, The Analysis of Linear Partial Differential Operators, Vol. III, Springer-Verlag, Berlin, 1985.

\bibitem{hale} J. K. Hale, 
Asymptotic Behavior of Dissipative Systems, 
Mathematical Surveys and Monographs 25, American Mathematical Society, Providence, 1988. 

\bibitem{hebey} E. Hebey, Nonlinear Analysis on Manifolds: Sobolev Spaces and Inequalities, Courant Lecture Notes in Mathematics 5, American Mathematical Society, Providence, 1999. 

\bibitem{joly-laurent} R. Joly and C. Laurent, {\it Stabilization for the semilinear wave equation with geometric control condition}, Analysis and PDE 6 (2013) 1089--1119. 

\bibitem{lady} O. Ladyzhenskaya, 
Attractors for Semigroups and Evolution Equations, 
Cambridge University Press, 1991. 

\bibitem{Lasiecka-Tataru} I. Lasiecka and D. Tataru, 
{\it Uniform boundary stabilization of semilinear wave equations with nonlinear boundary damping}, 
Differential Integral Equations 6 (1993) 507--533. 

\bibitem{LTZ} I. Lasiecka, R. Triggiani and X. Zhang, {\it Nonconservative wave equations with unobserved Neumann B.C.: global uniqueness and observability in one shot}, Differential geometric methods in the control of partial differential equations (Boulder, CO, 1999), 227--325, Contemp. Math., 268, Amer. Math. Soc., Providence, RI, 2000. 

\bibitem{MS1} R. B. Melrose and J. Sj\"ostrand, {\it Singularities of boundary value problems. I}, Comm. Pure Appl. Math. 31 (1978) 593--617.
	
\bibitem{MS2} R. B. Melrose and J. Sj\"ostrand, {\it Singularities of boundary value problems. II}, Comm. Pure Appl. Math. 35 (1982) 129--168.
	
\bibitem{miller} L. Miller, 
{\it Escape function conditions for the observation, control, and stabilization of the wave equation}, 
SIAM J. Control Optim. 41 (2003) 1554--1566.

\bibitem{ralston} J. V. Ralston, 
{\it Solutions of the wave equation with localized energy}, 
Comm. Pure Appl. Math. 22 (1969) 807--823. 

\bibitem{rt-cpam} J. Rauch and M. Taylor, 
{\it Decay of solutions to nondissipative hyperbolic systems on compact manifolds}, 
Comm. Pure Appl. Math. 28 (1975) 501--523. 

\bibitem{rt-indiana} J. Rauch and M. Taylor, 
{\it Exponential decay of solutions to hyperbolic equations in bounded domains}, 
Indiana Univ. Math. J. 24 (1974) 79--86. 

\bibitem{robbiano-zuily} L. Robbiano and C. Zuily, 
{\it Uniqueness in the Cauchy problem for operators with partially holomorphic coefficients}, 
Invent. Math. 131 (1998) 493--539. 

\bibitem{ruiz} A. Ruiz, 
{\it Unique continuation for weak solutions of the wave equation plus a potential}, 
J. Math. Pures Appl. (9) 71 (1992), no. 5, 455--467. 

\bibitem{seminario} P. N. Seminario--Huertas, Asymptotic dynamics of wave equations on compact Riemannian manifolds: 
sharp localized damping and supercritical forcing, Doctoral Dissertation, ICMC University of S\~ao Paulo, 2019.

\bibitem{stri} R. Strichartz, 
{\it Restriction of Fourier transform to quadratic surfaces and decay of solutions to the wave equation}, 
Duke Math. J. 44 (1977) 705--714. 

\bibitem{taylor} M. Taylor, Partial Differential Equations, volumes 1 and 2, Springer, Berlin, 1991. 

\bibitem{temam} R. Temam, 
Infinite-Dimensional Dynamical Systems in Mechanics and Physics, 
Applied Mathematical Sciences 68, Springer-Verlag, New York, 1997. 

\bibitem{ty2002} T. Triggiani and P. F. Yao, 
{\it Carleman estimates with no lower-order terms for general Riemann wave equations. Global uniqueness and observability in one shot}, 
Special issue dedicated to the memory of Jacques-Louis Lions, Appl. Math. Optim. 46 (2002) 331--375. 

\bibitem{peng} P. Yao, 
Modeling and Control in Vibrational and Structural Dynamics. A Differential Geometric Approach, 
Applied Mathematics and Nonlinear Science series, CRC press, Boca Raton, 2011. 

\end{thebibliography}
\end{document}